\documentclass[11pt,reqno]{amsart}
\usepackage{amsmath, amsthm, amssymb}
\usepackage{mathrsfs}
\usepackage{array}
\usepackage{caption}

\evensidemargin \oddsidemargin
\def\n{{\mathbf n}}

\def\bfn{{\boldsymbol n}}

\newtheorem{thm}{Theorem}

\newtheorem{lem}{Lemma}

\newtheorem{conj}{Conjecture}

\newtheorem{prop}{Proposition}

\numberwithin{equation}{section} \numberwithin{thm}{section}
\numberwithin{lem}{section} \numberwithin{problem}{section}
\numberwithin{cor}{section}

\newcommand{\norm}[1]{\left\Vert#1\right\Vert}

\begin{document}
\title{Mixed moments of the Riemann zeta function}
\author[Javier Pliego]{Javier Pliego}
\address{Department of Mathematics, KTH Royal Institute of Technology, Lindstedtsv\"agen 25,
10044 Stockholm, Sweden}

\email{javierpg@kth.se}
\subjclass[2010]{Primary 11M06; Secondary 11D75, 11J86}
\keywords{Riemann zeta function, moments of zeta, $abc$ conjecture, linear forms in logarithms}

\begin{abstract} We analyse a collection of mixed moments of the Riemann zeta function and establish the validity of asymptotic formulae. Such examinations are performed both unconditionally and under the assumption of a weaker version of the $abc$ conjecture.
\end{abstract}
\maketitle

\section{Introduction}
Investigations appertaining to the evaluation of moments of $L$-functions date back to Hardy-Littlewood \cite{Hard}, wherein an asymptotic formula for the second moment of the Riemann zeta function was established, it being followed by an analogous counterpart for the fourth moment due to Ingham \cite{Ing}. We thus define 
$$M_{k}(T)=\int_{0}^{T}\lvert \zeta(1/2+it)\rvert^{2k}dt$$ for $k\in\mathbb{N}$ and remark that subsequent work of numerous authors (see \cite{Bou6,Heath,Ivi,Zav}) sharpening the above results have led to formulae of the shape
\begin{equation}\label{Mkk}M_{k}(T)=TP_{k^{2}}(\log T)+O(T^{1-\delta})\ \ \ \ \ \ \ \ \ \ \end{equation}
for $k=1,2$, wherein $P_{k^{2}}(x)$ is a degree-$k^{2}$ polynomial and $0<\delta<1$ is a fixed constant. 

The extensive examinations concerning the asymptotic evaluation of higher moments have been on the contrary only conjectural, those being initiatied by Conrey-Ghosh \cite{Conr2} and Conrey-Gonek \cite{Conr3} in their papers analysing both the sixth and the eighth moment respectively. These were independently culminated with the incorporation of Random Matrix theory to the scene by Keating-Snaith \cite{Kea}, which delivered $M_{k}(T)\sim c_{k}T(\log T)^{k^{2}}$ for explicit constants $c_{k}$, such an especulation having been further refined in the work of Conrey et al. \cite{Conr} and taking the form (\ref{Mkk}) with $\delta=1/2-\varepsilon$ for $k\geq 3$. It also seems desirable to mention Soundararajan's \cite{Soun2} upper bound $M_{k}(T)\ll T(\log T)^{k^{2}+\varepsilon}$ conditional on the Riemann Hypothesis for all $k>0$, the $\varepsilon$ in the exponent ultimately being removed in a subsequent paper of Harper \cite{Harp}.

Little attention has been paid instead to the problem of analysing for $a,b,c \in\mathbb{R}^{+}$ and $T>1$ the integral
\begin{equation}\label{IABC}I_{a,b,c}(T)=\int_{0}^{T}\zeta(1/2+ait)\zeta(1/2-bit)\zeta(1/2-cit)dt.\end{equation} 
Investigating the above is partially motivated by the desire to examine a broader set of examples in search of similar phenomena to that occuring in a recent paper of Conrey-Keating (see \cite{Conr5}) in connection with the arithmetic stratification of subvarieties examined by Manin \cite{Man}. In the instance when the coefficients are integers then the error terms obtained may be substantially sharpened subjected to the validity of a weak version of the $abc$ conjecture (see \cite{mas}) that shall be presented shortly, the detailed discussion of which being deferred to a later point in the memoir.
\begin{conj}\label{conj11}
Let $\varepsilon>0$ and $a,b,c\in\mathbb{N}$ be fixed and $n_{1},n_{2},n_{3}\in\mathbb{N}$. Denote $$D=n_{1}^{a}-n_{2}^{b}n_{3}^{c}.$$ Then, if $D\neq 0$ one has the lower bound
$$\lvert D\rvert\gg n_{1}^{a-1-\varepsilon}n_{2}^{-1}n_{3}^{-1},$$ the implicit constant only depending on $\varepsilon$.
\end{conj} Equipped with the above considerations we announce some of the main results of the present memoir.
\begin{thm}\label{thm601}
Let $a,b,c\in\mathbb{N}$ with the property that $(a,b,c)=1$ and $a\geq 2$. Then, whenever $a<c\leq b$ one has for large $T$ the asymptotic relation
\begin{equation*}I_{a,b,c}(T)=\sigma_{a,b,c}T+E_{a,b,c}(T),\end{equation*} where $\sigma_{a,b,c}>1$ is a computable constant defined in (\ref{sigmaa}) and
\begin{equation}\label{jodo}E_{a,b,c}(T)\ll T^{1-1/2a+1/2c}(\log T)^{\eta},\end{equation} wherein $\eta=1$ if $b=c$ and $\eta=0$ else. If Conjecture \ref{conj11} holds, the above error term may be refined to
\begin{align*}E_{a,b,c}(T)\ll&
T^{1/2+a/(a+c)+\varepsilon}+T^{3/4}(\log T).\end{align*}

\end{thm}

It seems worth noting that the instance $(a,b,c)>1$ may be easily reduced to the above setting via a change of variables. We also remark that \cite{Bet2} yields $M_{3/2}(T)\asymp T(\log T)^{9/4}$ unconditionally, it thereby transpiring that in the context underlying Theorem \ref{thm601} further cancellation is exhibited. The relative simplicity of the off-diagonal analysis when $a=1$ enables one to refine the above result in this context and derive an analogous unconditional asymptotic formula.
\begin{thm}\label{prop3kr}
Let $b,c\in\mathbb{N}$ satisfying $1<c\leq b$. Then one has
$$I_{1,b,c}(T)=\zeta\big((1+b)/2\big)\zeta\big((1+c)/2\big)T+O\big(T^{3/4}(\log T)^{2}+T^{1/2+1/2b+\varepsilon}\big).$$ 
\end{thm}

The starting point to prove the above theorems shall make use of the approximate functional equation of the Riemman zeta function (see Titchmarsh \cite[(4.12.4)]{Tit}), namely
\begin{equation}\label{approxim}\zeta(1/2+it)=D_{1/2}(1/2+it)+\chi(1/2+it)D_{1/2}(1/2-it)+O(t^{-1/4}),\end{equation}
wherein \begin{equation}\label{chichi}\chi(s)=2^{s-1}\pi^{s}\sec(s\pi/2)\Gamma(s)^{-1},\ \ \ \ \ \ \ \ \ \ s\in\mathbb{C}\setminus (2\mathbb{Z}+1)\end{equation} 
and for $\theta>0$ the corresponding Dirichlet polynomial is defined by
\begin{equation}\label{diri}D_{\theta}(s)=\sum_{n\leq (t/2\pi)^{\theta}}n^{-s}\ \ \ \ \ \ \ \ \ \ s=\sigma+it,\end{equation} it therefore entailing the necessity of examining several integrals of twisted Dirichlet polynomials. In the investigation of the term with no twisting factor, one then makes a distinction between the diagonal and the off-diagonal contribution, the latter being substantially more problematic. Having a decent control of such an object then amounts to understanding the number of solutions of the equation
\begin{equation}\label{D1D2}n_{2}^{b}n_{3}^{c}-n_{1}^{a}=D\end{equation} for $D\in\mathbb{Z}$, the assumption of Conjecture \ref{conj11} delivering the lower bound \begin{equation}\label{D1D222}\lvert D\rvert\gg n_{1}^{a-1-\varepsilon}(n_{2}n_{3})^{-1}.\end{equation}

 Such a robust estimate improves the unconditional error terms derived, but does not suffice to enlarge the range of the parameters $a,b,c$ for which one may assure the validity of the asymptotic formula. It seems pertinent to mention that Sprindzuk \cite{Spr1,Spr2} on improving upon work of Baker \cite{Bak}, showed that whenever $n,m\in\mathbb{N}$ are fixed then one has $$\lvert x^{n}-y^{m}\rvert\gg (\log X)^{\delta(m,n)},\ \ \ \text{where $X=\max(x^{n},y^{m})$}$$ for some fixed $\delta(m,n)>0$. The above inequality lends credibility to the expectation that one may deduce unconditional estimates of the same strength for the difference $D$ in (\ref{D1D2}), but probably not better ones, the conditional bound (\ref{D1D222}) thereby seeming completely out of reach and illustrating the difficulty of the problem at hand. 

We surmount the difficulties associated with such an analysis by employing an argument only valid for the range of parameters $a<c\leq b$ that delivers an error term which may be refined if one assumes Conjecture \ref{conj11}. The perusal of the rest of the integrals containing twisting factors comprises estimates for oscillatory integrals. In the setting underlying Theorems \ref{thm601} and \ref{prop3kr} though there is an additional integral whose examination necessitates utilising a stationary phase method lemma of the strength of that of Graham and Kolesnik \cite[Lemma 3.4]{Gra} for the purpose of getting sharper error terms by exploiting some cancellation with other terms stemming from the diagonal contribution.

The arguments employed in the course of the proof herein involve an apparent unavoidable implicit use of the convexity bound, similar approaches to deliver asymptotic formulae for (\ref{IABC}) but with four zeta factors not being of sufficient robustness. The reader may have also wondered about the possibility of approximating each of the zeta factors instead by long Dirichlet polynomials, say of length $CT$ for some constant $C>0$. Such an avenue transports one to the problem of having to accurately analyse sums of the type
$$\sum_{\substack{n_{1},n_{2},n_{3}\leq CT}}(n_{1}n_{2}n_{3})^{-1/2}\log (n_{2}^{b}n_{3}^{c}/n_{1}^{a})^{-1}e^{iT\log (n_{2}^{b}n_{3}^{c}/n_{1}^{a})}$$ stemming from the off-diagonal contribution. If $\log (n_{2}^{b}n_{3}^{c}/n_{1}^{a})\asymp 1$, say, then when fixing $n_{2},n_{3}$ there are certain choices of $n_{1}$ having the property that $T/n_{1}$ has a small fractional part, whence the sum over integers not too far from such choices does not exhibit extra cancellation and provides a contribution of size $T^{3/2}$. Moreover, as outlined in (\ref{D1D2}), the poor understanding of the logarithm in the preceding equation precludes one from exploiting symmetries to estimate the above sum and makes the above approach not practicable. 

On a different note, we include a theorem concerning the asymptotic evalution of (\ref{IABC}) whenever $a=c<b$, there being a different behaviour underlying the anticipated formula.
\begin{thm}\label{thm602}
Let $a<b$ be positive integers satisfying $(a,b)=1$. Then one has that
\begin{equation*}I_{a,b,a}(T)\sim\zeta\big((a+b)/2\big)T\log T.\end{equation*}
\end{thm}
The examination of the off-diagonal contribution in this context departs from the analysis in the previous setting and exhibits some novelty, the estimates obtained having their reliance on an application of Roth's theorem in diophantine approximation \cite{Roth}. The ineffectiveness of the error term in the asymptotic formula at hand then stems from the ineffectiveness in such a theorem. Moreover, both the length $CT^{1/2}$ of the Dirichlet polynomials involved in (\ref{approxim}) and the corresponding exponent of $2+\varepsilon$ in the alluded theorem shall play a crucial role in the argument, to the extent that analogous inequalities for algebraic $\alpha\in\mathbb{R}$ of the shape
$$\lvert \alpha-a/q\rvert\geq \frac{C(\alpha)}{q^{2+\delta(\alpha)}},\ \ \ \ a\in\mathbb{Z}, q\in\mathbb{N}, \ \ \ \ \ \  C(\alpha),\delta(\alpha)>0,$$ 
containing effective information about the corresponding constant thereof at the cost of increasing the aforementioned exponent at hand (see \cite{Ben,Bomb,Fel}) shall not find success when applied in this setting. It is also noteworthy that such an application of Roth's theorem ultimately leads to a conclusion of the same strength than what would have been delivered under the assumption of Conjecture \ref{conj11}.

Much in the same vein as in the context underlying Theorem \ref{prop3kr}, the off-diagonal analysis is simpler when $a=1$ and does not require an appeal to results in diophantine approximation, such an alleviation enabling one to give account of lower order terms unconditionally. 
\begin{thm}\label{loweuse}
Let $b\in\mathbb{N}$ with $b>1$. Then with the above notation one has that
$$I_{1,b,1}(T)=\zeta\big((1+b)/2\big)T\log T+\nu_{b}T+O\big(T^{3/4}\log T\big)$$ if $b>2$, wherein $\nu_{b}\in\mathbb{R}$ is an explicit constant that shall be defined in (\ref{nunu}). When $b=2$ the same formula holds with an error term $O(T^{3/4+\varepsilon}).$ If instead $b=1$ then 
$$I_{1,1,1}(T)=\frac{1}{2}T(\log T)^{2}+c_{1}T\log T+c_{0}T+O\big(T^{3/4}\log T\big),$$ wherein the constants $c_{0},c_{1}\in\mathbb{R}$ will be made explicit in (\ref{c1c2c3}).
\end{thm}

On another note, one would expect that the functions $$\zeta(1/2+ait),\ \ \ \zeta(1/2+ib t),\ \ \ \zeta(1/2+ic t),\ \ \ \ \ \ a,b,c\in\mathbb{R}$$ with $a,b,c$ being linearly independent over $\mathbb{Q}$, are poorly correlated. Such a consideration may then lend credibility to the belief that the integral (\ref{IABC}) should exhibit substantially more cancellation in the case when the corresponding coefficients satisfy the preceding proviso. Confirming and quantifying this belief in the algebraic setting is, inter alia, the purpose of the upcoming theorem.

\begin{thm}\label{thm6025}
Let $a,b,c\in\mathbb{R}^{+}$ be algebraic numbers linearly independent over $\mathbb{Q}$ for which $a<c\leq b$. Then there is some effective constant $K_{a,b,c}>0$ such that
$$I_{a,b,c}(T)= T+O\big(Te^{-K_{a,b,c}(\log T)^{1/3}/(\log\log T)^{1/3}}\big).$$ 
\end{thm}
The case when the corresponding coefficients are in rational ratio may be easily reduced to the setting underlying Theorem \ref{thm601} via a change of variables, the only intermediate instance remaining to be analysed being that in which only one linear equation involving the coefficients holds. In this context we shall be confronted with a discussion concerning the signs of these coefficients, the nature of the ensuing asymptotic formula having its reliance on those.

\begin{thm}\label{thm1.5}
Let $a,b,c>0$ be algebraic numbers not in rational ratio for which $a<c\leq b$ and having the property that for some $l_{1},l_{2},l_{3}\in\mathbb{Z}$ satisfying $(l_{1},l_{2},l_{3})=1$ the equation
\begin{equation*}al_{1}=bl_{2}+cl_{3},\ \ \ \ \ \ \ \ \ (l_{1},l_{2},l_{3})\neq (0,0,0)\end{equation*} holds. Then if $l_{i}<0$ for some $1\leq i\leq 3$ one has that
$$I_{a,b,c}(T)=T+O\big(Te^{-K_{a,b,c}(\log T)^{1/3}/(\log\log T)^{1/3}}\big).$$ If on the contrary $l_{i}\geq 0$ for every $1\leq i\leq 3$ then 
$$I_{a,b,c}(T)=\zeta\big((l_{1}+l_{2}+l_{3})/2\big)T+O\big(Te^{-K_{a,b,c}(\log T)^{1/3}/(\log\log T)^{1/3}}\big).$$
\end{thm}

The proofs of Theorems \ref{thm6025} and \ref{thm1.5} primarily differ from that of Theorem \ref{thm601} in the simplicity of the diagonal contribution, the explanation of which having its reliance on a succinct application of Baker's theorem on linear forms in logarithms \cite{Bak1}. The analysis pertaining to the off-diagonal contribution shall also employ bounds stemming from such techniques \cite{Bak2}, those ultimately delivering weaker error terms primarily because of the absence of a strong spacing condition as in the integer setting.

We find it worth observing for the sake of completeness that an asymptotic formula in the simpler instance when the signs of the coefficients in (\ref{IABC}) are the same may be easily deduced with greater ease by employing some of the ideas in this memoir, it ultimately having the shape $$\int_{0}^{T}\zeta(1/2+ait)\zeta(1/2+bit)\zeta(1/2+cit)dt=T+O\big(T^{3/4}\log T\big)$$ for $a,b,c\in\mathbb{R}^{+}$. We have preferred to omit the proof of the preceding formula due to considerations of space.

The exposition of ideas is organised as follows. We begin by presenting some preliminary lemmata and showing how one may easily derive Conjecture \ref{conj11} from the $abc$ conjecture in Section \ref{preliminary}. Section \ref{offdiagonal} is then primarily devoted to the analysis of the off-diagonal contribution. The discussion concerning the diagonal contribution of the irrational algebraic coefficients is contained in Section \ref{secirr}. In Sections \ref{simple} and \ref{sec606} we routinarily bound oscillatory integrals involved in the formula at hand. In contrast, the analysis of the integral performed in Section \ref{sec605} involves the use of the stationary phase method. Section \ref{sec606} is completed with the proof of Theorems \ref{thm601}, \ref{prop3kr}, \ref{thm6025} and \ref{thm1.5}, and Theorems \ref{thm602} and \ref{loweuse} are then discussed and proved in Sections \ref{sec607} and \ref{hannah}. 

We write $[x]$ for $x\in\mathbb{R}$ to denote the nearest integer to $x$. Whenever $\varepsilon$ appears in any bound, it will mean that the bound holds for every $\varepsilon>0$, though the implicit constant then may depend on $\varepsilon$. We use $\ll$ and $\gg$ to denote Vinogradov's notation. When we employ such a notation to describe the limits of summation of a particular sum we shall only be interested in estimating such a sum, and the precise value of the implicit constant won't have any impact in the argument.

\emph{Acknowledgements}: The author's work was initiated during his visit to Purdue University under Trevor Wooley's supervision and finished at KTH while being supported by the G\"oran Gustafsson Foundation. The author would like to thank him for his guidance and helpful comments, Jonathan Bober for useful remarks and both Purdue University and KTH for their support and hospitality, and to acknowledge the activities supported by the NSF Grant DMS-1854398.

\section{Preliminary manoeuvres}\label{preliminary}

As a prelude to the analysis of integrals of unimodular functions, it has been thought convenient to include a sequel of lemmata preparing the ground for subsequent considerations. We start by recalling the following standard result concerning the asymptotic evaluation of the function $\chi(s)$ defined in (\ref{chichi}). 
\begin{lem}\label{lem601}
Let $t>0$. One then has
$$\chi(1/2+it)=\Big(\frac{2\pi}{t}\Big)^{it}e^{it+i\pi/4}\Big(1+O\Big(\frac{1}{t}\Big)\Big),\ \ \ \chi(1/2-it)=\Big(\frac{2\pi}{t}\Big)^{-it}e^{-it-i\pi/4}\Big(1+O\Big(\frac{1}{t}\Big)\Big)$$ as $t\to\infty$.
\end{lem}
\begin{proof} The above formulae follow from the equation right after Titchmarsh \cite[(7.4.3)]{Tit} containing the asymptotic evaluation of $\chi(1-s).$
\end{proof}

We next shift our attention to the discussion concerning the fact that the $abc$ conjecture implies Conjecture \ref{conj11}, it being worth to this end stating such a conjecture (see \cite{mas}).
\begin{conj}[$abc$ conjecture]\label{conjert}
Given $\varepsilon>0$ there exists a constant $C_{\varepsilon}$ with the property that for every triple of coprime non-zero integers $(a,b,c)$ satisfying $a+b=c$, one has that
$$\max(\lvert a\rvert,\lvert b\rvert,\lvert c\rvert)\leq C_{\varepsilon}\Big(\prod_{p\mid (abc)}p\Big)^{1+\varepsilon}.$$
\end{conj}
Equipped with the above statement we present the following consequence of its assumption.
\begin{lem}
Conjecture \ref{conjert} implies Conjecture \ref{conj11}.
\end{lem}
\begin{proof}
We take a triple of natural numbers $n_{1},n_{2},n_{3}\in\mathbb{N}$, introduce the number $$D=n_{1}^{a}-n_{2}^{b}n_{3}^{c},$$ recall that $D\neq 0$ as was assumed in the statement of Conjecture \ref{conj11}, and write for convenience $$\lambda=\text{gcd}(\lvert D\rvert,n_{1}^{a},n_{2}^{b}n_{3}^{c}).$$ Observe that then the triple $$(N_{1},N_{2},N_{3})=( n_{1}^{a}\lambda^{-1},n_{2}^{b}n_{3}^{c}\lambda^{-1},\lvert D\rvert\lambda^{-1})$$ comprises relatively coprime positive integers in view of the proviso $D\neq 0$, whence an application of Conjecture \ref{conjert} delivers the inequality $$N_{1}\leq \max(N_{1},N_{2},N_{3})\ll \Big(\prod_{p|(N_{1}N_{2}N_{3})}p\Big)^{1+\varepsilon}.$$ It may be worth observing that
$$\prod_{p|(N_{1}N_{2}N_{3})}p\ll \prod_{p|(n_{1}^{a}n_{2}^{b}n_{3}^{c}D\lambda^{-1})}p= \prod_{p|(n_{1}n_{2}n_{3}D\lambda^{-1})}p\ll n_{1}n_{2}n_{3}\lvert D\rvert\lambda^{-1},$$ the equality in the above line having its reliance on the fact that $D/\lambda$ is an integer, whence a combination of the above equations yields the desired conclusion.

\end{proof}

We conclude by demonstrating how the problem can be reduced to that of computing integrals of products of twisted Dirichlet polynomials, it being desirable to recall first for each $0<\theta<1$ the definition of $D_{\theta}(s)$ in (\ref{diri}). We write $D(s)$ to denote $D_{1/2}(s)$ for the sake of simplicity and introduce
\begin{equation*}P(t)=D(1/2+it)+\chi(1/2+it)D(1/2-it)\end{equation*} for $t\in\mathbb{R},$ where $\chi(s)$ was defined in (\ref{chichi}). It shall also be convenient for further use to define for $T>0$ and fixed $\theta\in\mathbb{R}$ and $a\in\mathbb{N}$ the parameters
\begin{equation}\label{QQQQ}T_{1}=T/2\pi,\ \ \ \ \ \ \ Q=\Big(\frac{aT}{2\pi}\Big)^{\theta},\end{equation} the absence of the dependence with respect to $a,\theta$ in the notation having been imposed for the sake of simplicity.
\begin{lem}\label{lemita601}
With the above notation, one has
\begin{equation}\label{IT5.1}I_{a,b,c}(T)=\int_{0}^{T}P(at)P(-bt)P(-ct)dt+O\big(T^{3/4}(\log T)\big).\end{equation} 
\end{lem}
\begin{proof}

We begin by recalling the approximate functional equation (\ref{approxim}) and defining for $n\in\mathbb{Z}$ the function $\zeta_{n}(t)=\zeta(1/2+nit).$ By using such a formula one readily sees that
\begin{equation*}I_{a,b,c}(T)=\int_{0}^{T}P(at)P(-bt)P(-ct)dt+E(T),\end{equation*}wherein the error term $E(T)$ in the above line satisfies the estimate$$E(T)\ll T^{1/4}+E_{1}(T)+E_{2}(T),$$ and the terms $E_{1}(T)$ and $E_{2}(T)$ are defined by the relations
$$E_{1}(T)= \int_{1}^{T}t^{-1/2}\big(\lvert\zeta_{a}(t)\rvert+\lvert\zeta_{-b}(t)\rvert+\lvert\zeta_{-c}(t)\rvert\big) dt$$ and
$$E_{2}(T)=\int_{1}^{T}t^{-1/4}\Big(\lvert\zeta_{a}(t)\rvert \lvert\zeta_{-b}(t)\rvert +\lvert\zeta_{a}(t)\rvert\lvert\zeta_{-c}(t)\rvert+\lvert\zeta_{-b}(t)\rvert\lvert\zeta_{-c}(t)\rvert\Big)dt.$$We use Cauchy's inequality in conjunction with the asymptotic formula for the second moment of the Riemman Zeta function (see Titchmarsh \cite[Theorem 7.3]{Tit}) to obtain
$$E_{1}(T)\ll (\log T)^{1/2}\Big(\int_{0}^{T}\lvert\zeta(1/2+it)\rvert^{2}\Big)^{1/2}\ll T^{1/2}\log T.$$
Likewise, integration by parts combined with another application of Cauchy's inequality and the aforementioned formula delivers
\begin{align*}E_{2}(T)&\ll T^{-1/4}\int_{0}^{T}\lvert\zeta(1/2+it)\rvert^{2}+\int_{1}^{T}t^{-5/4}\int_{0}^{t}\lvert\zeta(1/2+is)\rvert^{2}dsdt
\\
&\ll T^{3/4}\log T+\int_{1}^{T}t^{-1/4}\log tdt\ll T^{3/4}(\log T).\end{align*} The combination of the above estimates yields the desired result.

\end{proof}

In view of Lemma \ref{IT5.1} it transpires that establishing the asymptotic evaluations which we seek to deliver amounts to computing several integrals of products of twisted Dirichlet polynomials. To the end of not providing the definitions of such integrals all at once it has been thought preferable to define those right before stating each of the lemma concerning their analysis. We shall thereby express the integral in (\ref{IT5.1}) as a sum of six terms, namely
\begin{equation}\label{piresiti}I_{a,b,c}(T)=\sum_{j=1}^{6}I_{j}(T) +O\big(T^{3/4}(\log T)\big).\end{equation}

\section{Off-diagonal contribution}\label{offdiagonal}
We shall devote the present section to the analysis of the off-diagonal contribution, it being pertinent to such an end to introduce beforehand some notation which will be used henceforth in the memoir. We fix a triple $(a,b,c)\in\mathbb{R}^{+}$ and $\theta\in\mathbb{R}$ satisfying $0<\theta<1$, and omit writing the implicit constants' dependence on such a triple underlying the estimates throughout the paper. We write $\bfn=(n_{1},n_{2},n_{3})\in\mathbb{N}^{3}$, and consider the weighted variables \begin{equation}\label{n1prima}n_{1,\theta}'=n_{1}/a^{\theta},\ \ \ \ n_{2}'=n_{2}/\sqrt{b},\ \ \ \ n_{3}'=n_{3}/\sqrt{c},\end{equation} and the parameters \begin{equation}\label{NNN}N_{\bfn,\theta}=2\pi\max(n_{1,\theta}'^{1/\theta},n_{2}'^{2},n_{3}'^{2}) \ \ \ \ \ \text{and}\ \ \ \ \ \  P_{\bfn}=n_{1}n_{2}n_{3}.\end{equation} We write as is customary $n_{1}'$ and $N_{\bfn}$ to denote $n_{1,1/2}'$ and $N_{\bfn,1/2}$ respectively for the sake of concision. We recall (\ref{QQQQ}) and foreshadow the convenience of introducing the set of triples
\begin{equation}\label{Babc}\mathcal{B}_{a,b,c,\theta}=\Big\{\bfn\in\mathbb{N}^{3}:\ \ n_{1}\leq Q,\ \ n_{2}\leq \sqrt{b T_{1}},\ \ n_{3}\leq \sqrt{c T_{1}}\Big\},\end{equation} the condition $\bfn \in \mathcal{B}_{a,b,c,\theta}$ being equivalent to the inequality
\begin{equation*}\label{lades}N_{\bfn,\theta}\leq T,\end{equation*}and the letter $\mathcal{B}_{a,b,c}$ denoting  $\mathcal{B}_{a,b,c,1/2}.$

As a prelude to our discussion we present to the reader 
\begin{align*}\label{Is1}I_{1,\theta}(T)&
=\int_{0}^{T}D_{\theta}(1/2+ait)D(1/2-bit)D(1/2-cit)dt=\sum_{\bfn\in \mathcal{B}_{a,b,c,\theta}}P_{\bfn}^{-1/2}\int_{N_{\bfn}}^{T}\Big(\frac{n_{2}^{b}n_{3}^{c}}{n_{1}^{a}}\Big)^{it}dt,\end{align*}the notation $I_{1,1/2}(T)$ being abbreviated by means of $I_{1}(T)$. We obtain throughout the memoir results for the more general collection of Dirichlet polynomials $D_{\theta}(1/2+ait)$ for future use, the case $\theta=1/2$ being the only required one herein. We make a distinction between the diagonal and the off-diagonal contribution to obtain
\begin{equation}\label{I131}I_{1,\theta}(T)=J_{1}^{\theta}(T)+J_{2}^{\theta}(T),\end{equation}where in the above equation one has \begin{equation}\label{0.121}J_{1}^{\theta}(T)=\sum_{\substack{\bfn\in \mathcal{B}_{a,b,c,\theta}\\ n_{1}^{a}=n_{2}^{b}n_{3}^{c}}}(T-N_{\bfn,\theta})P_{\bfn}^{-1/2},\ \ \ \ \ \ \ \ \ J_{2}^{\theta}(T)=\sum_{\substack{\bfn\in \mathcal{B}_{a,b,c,\theta}\\ n_{1}^{a}\neq n_{2}^{b}n_{3}^{c}}}P_{\bfn}^{-1/2}\int_{N_{\bfn}}^{T}\Big(\frac{n_{2}^{b}n_{3}^{c}}{n_{1}^{a}}\Big)^{it}dt.\end{equation}We write $J_{1}(T)$ and $J_{2}(T)$ as is customary to denote $J_{1}^{1/2}(T)$ and $J_{2}^{1/2}(T)$ respectively. For ease of notation we may omit writing $\bfn\in \mathcal{B}_{a,b,c,\theta}$ in the subscripts of the sums throughout the rest of the memoir and begin by analysing first $J_{2}^{\theta}(T)$. It then may seem appropiate to define for convenience
\begin{equation}\label{Nbc}N_{1}=[ n_{2}^{b/a}n_{3}^{c/a}]\end{equation} for each tuple $(n_{2},n_{3})$. We split the corresponding sum accordingly to obtain \begin{equation}\label{effi}J_{2}^{\theta}(T)= J_{2,2}(T)+O\big(J_{2,1}(T)\big),\end{equation} where
\begin{equation}\label{J212}J_{2,1}(T)=\sum_{\substack{\bfn\in \mathcal{B}_{a,b,c,\theta}\\ n_{1}\neq N_{1}}}P_{\bfn}^{-1/2}\lvert\log\big(n_{1}^{a}/n_{2}^{b}n_{3}^{c}\big)\rvert^{-1}\end{equation} and \begin{equation}\label{J22}J_{2,2}(T)=\sum_{\substack{\bfn\in \mathcal{B}_{a,b,c,\theta}\\ n_{1}=N_{1}}}P_{\bfn}^{-1/2}\int_{N_{\bfn,\theta}}^{T}e^{it\log(n_{2}^{b}n_{3}^{c}/n_{1}^{a})}dt.\end{equation}
We shall omit throughout the paper and as above writing the condition $n_{1}^{a}\neq n_{2}^{b}n_{3}^{c}$ in every sum cognate to the off-diagonal contribution. The reader may find it worth observing that whenever $n_{1}\asymp N_{1}$ then \begin{equation}\label{logilogi}\lvert\log\big(n_{2}^{b}n_{3}^{c}/n_{1}^{a}\big)\rvert\asymp \frac{\lvert n_{2}^{b}n_{3}^{c}-n_{1}^{a}\rvert}{n_{2}^{b}n_{3}^{c}}\asymp \frac{\lvert n_{2}^{b/a}n_{3}^{c/a}-n_{1}\rvert}{n_{2}^{b/a}n_{3}^{c/a}}.\end{equation}

It then transpires that in view of the above relation one may deduce the estimate \begin{align}\label{N1aa}\sum_{\substack{\frac{N_{1}}{2}\leq n_{1}<2N_{1}\\ n_{1}\neq N_{1}}}\frac{n_{1}^{-1/2}}{\lvert\log\big(n_{2}^{b}n_{3}^{c}/n_{1}^{a}\big)\rvert}
\ll \sum_{1\leq \lvert r\rvert\leq N_{1}}\frac{N_{1}^{1/2}}{\lvert n_{2}^{b/a}n_{3}^{c/a}-N_{1}-r\rvert}\ll \sum_{1\leq r\leq N_{1}}\frac{N_{1}^{1/2}}{r}\ll N_{1}^{1/2}\log T.
\end{align}
It seems appropiate to denote $J_{2,1,1}(T)$ the contribution to $J_{2,1}(T)$ of tuples in the range considered in the above line. We recall (\ref{QQQQ}) and observe that when $a\leq c\leq b$ then the preceding equation in conjunction with (\ref{Nbc}) and (\ref{J212}) delivers
\begin{align*}J_{2,1,1}(T)&
\ll (\log T)\sum_{n_{2}^{b}n_{3}^{c}\ll Q^{a}}n_{2}^{-1/2+b/2a}n_{3}^{-1/2+c/2a}\nonumber
\\
&\ll Q^{1/2+a/2b}(\log T)\sum_{n_{3}\ll Q^{a/c}}n_{3}^{-1/2-c/2b}\ll Q^{1/2+a/2c}(\log T)\ll Q^{1/2}T^{1/2}.
\end{align*}
It is worth noting that whenever $n_{1}$ is outside of the range considered above then it transpires that $\lvert \log\big(n_{1}^{a}/n_{2}^{b}n_{3}^{c}\big)\rvert\geq C$ for some positive constant $C>0$, the contribution stemming from these cases thereby being $O(Q^{1/2}T^{1/2})$. Therefore, the preceding discussion yields
\begin{equation}\label{bJ21}J_{2,1}(T)\ll Q^{1/2}T^{1/2}.\end{equation}

The following proposition will be devoted to the estimation of $J_{2}^{\theta}(T)$ for the case of algebraic real coefficients, it being pertinent to present to the reader a version of a theorem of Baker on linear forms in logarithms that shall be utilised in the proof.
\begin{thm}\label{linearbaker}
Let $\beta_{1},\ldots,\beta_{n}$ be a collection of algebraic numbers, and let $\alpha_{1},\ldots,\alpha_{n}\in \mathbb{Q}^{+}$, non of which being $1$. Assume that $\alpha_{j}\leq A_{j}$ for fixed positive numbers $A_{j}$. Then on defining $\Omega=\log A_{1}\cdots\log A_{n}$ and $\Omega'=\Omega/\log A_{n}$ and
$$\Lambda=\beta_{1}\log \alpha_{1}+\ldots+\beta_{n}\log \alpha_{n},$$ one has that if $\Lambda\neq 0$ then $$\lvert\Lambda\rvert>e^{-C\Omega\log \Omega'},$$ wherein the constant $C=C(\beta_{1},\ldots,\beta_{n})$ is effective.
\end{thm}
\begin{proof}
The proof follows from Baker \cite{Bak2}.
\end{proof}
Equipped with the above result and upon recalling (\ref{QQQQ}), we are prepared to analyse the off-diagonal contribution in the aforementioned setting.

\begin{lem}\label{prop4789}
Let $a,b,c\in\mathbb{R}$ be positive algebraic numbers satisfying $a<c\leq b$. Then,
$$J_{2}(T)\ll Te^{-K_{a,b,c}(\log T)^{1/3}/(\log\log T)^{1/3}},$$wherein the constant $K_{a,b,c}>0$ is effective and only depends on $a,b,c$.
\end{lem}
\begin{proof}

We shall employ (\ref{effi}) and (\ref{bJ21}) and reduce our task to the analysis of $J_{2,2}(T)$.
 It is apparent that an application of Theorem \ref{linearbaker} enables one to derive 
\begin{equation}\label{dosde}\lvert\log\big(n_{2}^{b}n_{3}^{c}/N_{1}^{a}\big)\rvert>e^{-C(\log N_{1})^{3}\log\log N_{1}}\end{equation} for some constant $C>0$ only depending on $a,b,c$. We dropped such a dependance from the notation for the sake of brevity. We also introduce for some small enough constant $\delta>0$ only depending on the coefficients $a,b,c$ the function $$G_{\delta}(t)=e^{\delta(\log t)^{1/3}/(\log\log t)^{1/3}}.$$ Equipped with such a bound, we divide the range of summation to obtain
$$J_{2,2}(T)\ll F_{1}(T)+F_{2}(T),$$ where the above terms are defined by means of the formulae
$$  F_{1}(T)= \sum_{N_{1}\leq G_{\delta}(T)}n_{2}^{-1/2-b/2a}n_{3}^{-1/2-c/2a}e^{C(\log N_{1})^{3}\log\log N_{1}},$$ such an estimate indeed stemming from the application of (\ref{dosde}), and $$F_{2}(T)=T\sum_{N_{1}> G_{\delta}(T)}n_{2}^{-1/2-b/2a}n_{3}^{-1/2-c/2a}.$$

Then by summing over $n_{2}$ and $n_{3}$ accordingly one obtains
$$F_{2}(T)\ll TG_{\delta}(T)^{a/2c-1/2}(\log T),$$ wherein the reader may want to recall the condition $a<c$. Moreover, it transpires that whenever $N_{1}\leq G_{\delta}(T)$ then
$$e^{C(\log N_{1})^{3}\log\log N_{1}}\ll T^{\delta'}$$ for some small enough constant $\delta'>0$, such an observation thus yielding the bound
$$F_{1}(T)\ll T^{\delta'}.$$
The combination of the above estimates then delivers the desired result.
\end{proof}

\begin{lem}\label{prop478}
If $a,b,c\in\mathbb{N}$ satisfying $a<c\leq b$ then for $0<\theta<1$ the bounds
$$J_{2,2}(T)\ll T^{1+1/2c-1/2a}(\log T)^{\eta},\ \ \ \ \ \ \ J_{2}^{\theta}(T)\ll Q^{1/2}T^{1/2}+T^{1+1/2c-1/2a}(\log T)^{\eta},$$ wherein $\eta$ was defined after (\ref{jodo}), hold unconditionally. If one assumes Conjecture \ref{conj11} then 
$$J_{2,2}(T)\ll Q^{1/2+3a/2c+\varepsilon},\ \ \ \ \ \ \ \ J_{2}^{\theta}(T)\ll Q^{1/2}T^{1/2}+Q^{1/2+3a/2c+\varepsilon}.$$

\end{lem}
\begin{proof}
We shall make use of (\ref{effi}) and (\ref{bJ21}) as above and modify the analysis of $J_{2,2}(T)$ to the end of deriving sharper estimates, it being desirable to remark that triples underlying the term $J_{2,2}(T)$ satisfy \begin{equation}\label{spaci1}\lvert n_{2}^{b}n_{3}^{c}-N_{1}^{a}\rvert\geq 1.\end{equation} We divide the range of summation in accordance with the first inequality in (\ref{logilogi}) to obtain 
$$J_{2,2}(T)\ll F_{1}(T)+F_{2}(T),$$ where the preceding terms are defined by means of the formulas
$$  F_{1}(T)=\sum_{n_{2}^{b}n_{3}^{c}\leq T}n_{2}^{b-1/2-b/2a}n_{3}^{c-1/2-c/2a},\ \ \ \ \ F_{2}(T)=T\sum_{n_{2}^{b}n_{3}^{c}> T}n_{2}^{-1/2-b/2a}n_{3}^{-1/2-c/2a}.$$ It seems worth noting that we bounded the integral in (\ref{J22}) by the inverse of the corresponding logarithm and applied (\ref{logilogi}) and (\ref{spaci1}) subsequently, the same integral cognate to the term $F_{2}(T)$ being estimated by the length of the interval of integration. Summing then over $n_{2}$ first yields
$$F_{1}(T)\ll T^{1+1/2b-1/2a}\sum_{n_{3}\leq T^{1/c}}n_{3}^{-1/2-c/2b}\ll T^{1+1/2c-1/2a}(\log T)^{\eta}.$$ Likewise, an analogous computation reveals that
$$F_{2}(T)\ll T^{1+1/2b-1/2a}\sum_{n_{3}^{c}\leq T}n_{3}^{-1/2-c/2b}+T\sum_{n_{3}^{c}> T}n_{3}^{-1/2-c/2a}\ll T^{1+1/2c-1/2a}(\log T)^{\eta},$$
whence the above estimates deliver
 $$J_{2,2}(T)\ll T^{1+1/2c-1/2a}(\log T)^{\eta},$$ which combined with (\ref{bJ21}) yields the desired conclusion. 

The assumption of Conjecture \ref{conj11} in conjunction with (\ref{logilogi}) yields for tuples associated to the aforementioned term the bound
$$\big\lvert\log \big(N_{1}^{a}/n_{2}^{b}n_{3}^{c})\big\rvert^{-1}\ll (N_{1}n_{2}n_{3})^{1+\varepsilon},$$ whence inserting the previous estimate in (\ref{J22}) one gets
\begin{align*}J_{2,2}(T)\ll T^{\varepsilon}\sum_{n_{2}^{b}n_{3}^{c}\ll Q^{a}}n_{2}^{1/2+b/2a}n_{3}^{1/2+c/2a}&
\ll Q^{1/2+3a/2c+\varepsilon}\sum_{n_{2}\ll Q^{a/b}}n_{2}^{1/2-3b/2c}
\\
&\ll Q^{1/2+3a/2c+\varepsilon}.
\end{align*}
The lemma follows by combining the above bound with (\ref{bJ21}).

\end{proof}
It seems worth pointing out that the assumption $a<\min(b,c)$ was crucially utilised in both the unconditional analysis and the conditional one whenever $\theta\geq 1/2$, an analogue argument not being applicable in other circumstances. If $a=1$ and $\theta=1/2$ then we may obtain an unconditional sharper estimate which would ultimately deliver an error term $O(T^{1/4+1/4c})$ by making use of the corresponding inequality $n_{2}^{b}n_{3}^{c}\ll T^{1/2}$ in the above setting, such a refinement not having any impact in the conclusion of Theorem \ref{prop3kr}. To make further progress it then seems worth recalling the definitions in (\ref{n1prima}) and the subsequent lines and writing
\begin{equation}\label{J1J9}J_{1}(T)=\sigma_{a,b,c}T-J_{3}(T)-J_{4}(T),\end{equation}
wherein the above terms are \begin{equation}\label{0.1212}J_{3}(T)=T\sum_{\substack{n_{1}^{a}=n_{2}^{b}n_{3}^{c}\\ \max(n_{1}'^{2},n_{2}'^{2},n_{3}'^{2})>T_{1}}}P_{\bfn}^{-1/2},\ \ \ \ \ \ \ J_{4}(T)=\sum_{\substack{\bfn\in \mathcal{B}_{a,b,c}\\ n_{1}^{a}=n_{2}^{b}n_{3}^{c}}}N_{\bfn}P_{\bfn}^{-1/2},\end{equation}
and wherein the constant $\sigma_{a,b,c}$ is defined by means of the formula
\begin{equation}\label{sigmaa}\sigma_{a,b,c}=\sum_{n_{1}^{a}=n_{2}^{b}n_{3}^{c}}P_{\bfn}^{-1/2},\end{equation} the convergence of the above series being delivered inter alia by the fact that $a<\min(b,c)$. We observe that a combination of the previous equations with (\ref{I131}), (\ref{effi}) and (\ref{bJ21}) for the choice $\theta=1/2$ enables one to derive
\begin{equation}\label{osss}I_{1}(T)=\sigma_{a,b,c}T-J_{3}(T)-J_{4}(T)+J_{2,2}(T)+O(T^{3/4}),\end{equation} the application of Lemma \ref{prop478} thus entailing 
\begin{equation}\label{osssa}I_{1}(T)=\sigma_{a,b,c}T-J_{3}(T)-J_{4}(T)+O\big(T^{3/4}+T^{1+1/2c-1/2a}(\log T)^{\eta}\big)\end{equation} unconditionally and \begin{equation}\label{osssa1}I_{1}(T)=\sigma_{a,b,c}T-J_{3}(T)-J_{4}(T)+O\big(T^{3/4}+T^{1/4+3a/4c+\varepsilon}\big)\end{equation} under the assumption of Conjecture \ref{conj11}.

\section{Analysis of the irrational algebraic case}\label{secirr}

We shall begin our discussion with an application of Baker's theorem \cite{Bak1} on linear forms in logarithms to discard the existence of non-trivial diagonal solutions of the underlying equation whenever the coefficients are linearly independent over $\mathbb{Q}$.
\begin{prop}\label{diagono}
Let $a,b,c\in\mathbb{R}^{+}$ be algebraic numbers linearly independent over $\mathbb{Q}$. Then there are no solutions to the equation
\begin{equation}\label{abcabc}n_{1}^{a}=n_{2}^{b}n_{3}^{c},\ \ \ \ \ \ \ n_{1},n_{2},n_{3}\in\mathbb{N},\ \ \ \ (n_{1},n_{2},n_{3})\neq (1,1,1).\end{equation}
Consequently, there is an effective constant $K_{a,b,c}$ only depending on $a,b,c$ such that
$$I_{1}(T)= T+O\big(Te^{-K_{a,b,c}(\log T)^{1/3}/(\log\log T)^{1/3}}\big),$$ 
\end{prop}
\begin{proof}

The second statement follows from the first by recalling (\ref{I131}) and (\ref{0.121}) and noting that then $$J_{1}(T)=T-2\pi a^{-1},$$ which in conjunction with Lemma \ref{prop4789} delivers the desired result.

We shift our focus to the first assertion then and begin by assuming the existence of a triple $(n_{1},n_{2},n_{3})$ with the property (\ref{abcabc}). In particular, this entails the linearly dependence of $\log n_{1},\log n_{2}, \log n_{3}$ and $2\pi i$ over the algebraic numbers. Therefore, an application of Baker's theorem \cite{Bak1} establishes the linearly dependence over $\mathbb{Q}$, which in turn yields the existence of rational numbers $r_{1},r_{2},r_{3}\in\mathbb{Q}$ satisfying
$$n_{1}^{r_{1}}=n_{2}^{r_{2}}n_{3}^{r_{3}},$$ wherein we may assume without loss of generality that $r_{1}\neq 0$. Moreover, an examination of (\ref{abcabc}) and the preceding equation reveals that \begin{equation}\label{n2no1}(n_{2},n_{3})\neq (1,1).\end{equation}Therefore, combining both of the equations then delivers the relation
\begin{equation}\label{linearly44}n_{2}^{b/a-r_{2}/r_{1}}n_{3}^{c/a-r_{3}/r_{1}}=1.\end{equation}

It seems pertinent to observe that in view of the linear independence over the rationals of the coefficients $a,b,c$ then the exponents in the above line are non-zero. It therefore transpires by (\ref{n2no1}) and the preceding observation that then $n_{2}\neq 1\neq n_{3}.$ We apply, as we may, Baker's theorem \cite{Bak1} again to obtain a rational number $r_{4}\in\mathbb{Q}$ having the property $$n_{2}^{r_{4}}=n_{3}.$$
Combining the above line with that of (\ref{linearly44}) yields an equality between the corresponding exponents, namely,
$$b+cr_{4}=\Big(\frac{r_{4}r_{3}+r_{2}}{r_{1}}\Big)a,$$ which contradicts the linear independence of the coefficients.

\end{proof}

\begin{prop}\label{prop55}

Let $a,b,c\in\mathbb{R}^{+}$ such that $a<c\leq b$ be algebraic numbers linearly dependent over $\mathbb{Q}$ and not in rational ratio. Let $l_{1},l_{2},l_{3}\in\mathbb{Z}$ satisfying $(l_{1},l_{2},l_{3})=1$ with the property that 
\begin{equation}\label{relatii}al_{1}=bl_{2}+cl_{3},\ \ \ \ \ \ \ \ \ (l_{1},l_{2},l_{3})\neq (0,0,0).\end{equation} Then if $l_{i}<0$ for some $i\leq 3$ one has that
$$I_{1}(T)=T+O\big(Te^{-K_{a,b,c}(\log T)^{1/3}/(\log\log T)^{1/3}}\big).$$ If on the contrary $l_{i}\geq 0$ for every $1\leq i\leq 3$ then 
$$I_{1}(T)=\zeta\big((l_{1}+l_{2}+l_{3})/2\big)T+O\big(Te^{-K_{a,b,c}(\log T)^{1/3}/(\log\log T)^{1/3}}\big).$$
\end{prop}

\begin{proof}
We proceed as above and assume the existence of a triple $(n_{1},n_{2},n_{3})\in\mathbb{N}^{3}$ such that \begin{equation}\label{n1n2n3}n_{1}^{a}=n_{2}^{b}n_{3}^{c},\ \ \ \ \ \ (n_{1},n_{2},n_{3})\neq (1,1,1).\end{equation} The reader may observe that in view of the above equation it is apparent that $(n_{2},n_{3})\neq (1,1)$ as well. An application of Baker's theorem \cite{Bak1} then yields the existence of rational numbers $r_{1},r_{2},r_{3}\in\mathbb{Q}$ satisfying
\begin{equation}\label{r1r2}n_{1}^{r_{1}}=n_{2}^{r_{2}}n_{3}^{r_{3}},\end{equation} with at least one non-zero exponent, say $r_{1}\neq 0$. Therefore, combining both of the equations then delivers the relation
\begin{equation}\label{linearly4}n_{2}^{b/a-r_{2}/r_{1}}n_{3}^{c/a-r_{3}/r_{1}}=1,\end{equation}the above exponents not being simultaneously $0$ in view of the fact that $a,b,c$ are not in rational ratio. If $n_{2}\neq 1\neq n_{3}$ then another application of Baker's theorem \cite{Bak1} assures the validity of the equation $n_{2}^{r_{4}}=n_{3}$ for some $r_{4}\in\mathbb{Q}^{+}$. Therefore, by the preceding discussion it transpires that there exists a triple $(s_{1},s_{2},s_{3})\in\mathbb{N}^{3}$ satisfying the proviso $(s_{1},s_{2},s_{3})=1$ and a natural number $m\in\mathbb{N}$ such that $m\neq 1$ and for which
$$n_{1}=m^{s_{1}},\ \ \ \ \ \ n_{2}=m^{s_{2}},\ \ \ \ \ \ n_{3}=m^{s_{3}}.$$
Combining the above observation with (\ref{n1n2n3}) delivers the relation
$s_{1}a=s_{2}b+s_{3}c.$ The coprimality condition earlier described in conjunction with the fact that the coefficients $a,b,c$ are not in rational ratio then yields
$$(s_{1},s_{2},s_{3})=(l_{1},l_{2},l_{3}).$$

If the assumption made right after (\ref{linearly4}) does not hold, we then may suppose without loss of generality that $n_{2}=1$, whence (\ref{r1r2}) assures as is customary the existence of a natural number $m\geq 2$ and a pair $s_{1},s_{3}\in\mathbb{N}$ satisfying $(s_{1},s_{3})=1$ and with the property that
$$n_{1}=m^{s_{1}},\ \ \ \ \ \ \ \ \ n_{3}=m^{s_{3}}.$$ Combining the above relation with (\ref{n1n2n3}) then provides the identity
$s_{1}a=s_{3}c.$ Consequently, the coprimality condition earlier mentioned and the corresponding irrational ratio property enable one to deduce that
$$(s_{1},0,s_{3})=(l_{1},l_{2},l_{3}).$$
The preceding discussion thus assures that for either circumstance the only solutions of the equation (\ref{n1n2n3}) are of the shape $$n_{1}=m^{l_{1}},\ \ \ \ \ \ n_{2}=m^{l_{2}},\ \ \ \ \ \ n_{3}=m^{l_{3}},\ \ \ \ \ \ \ \ \ \ \ \ m\in\mathbb{N}.$$ 

In particular, it seems desirable to remark that as a consequence of the ensuing analysis, there are no solutions to (\ref{n1n2n3}) when $l_{i}<0$ for some $1\leq i\leq 3,$ whence on recalling (\ref{J1J9}) one has that $$J_{1}(T)=T-2\pi a^{-1},$$ which combined with Proposition \ref{prop4789} delivers the desired result for this particular instance. For the remaining cases we recall (\ref{0.1212}) and observe that then
$$J_{4}(T)\ll \sum_{m^{l_{1}}\leq \sqrt{aT_{1}}}m^{3l_{1}/2-(l_{2}+l_{3})/2}\ll T^{3/4+(2-l_{2}-l_{3})/4l_{1}}.$$ We draw the reader's attention to (\ref{relatii}) to the end of deducing that in view of the fact that $a<\min(b,c)$ it transpires that $l_{1}\geq 2$ and that $l_{2}+l_{3}\geq 1$, such a remark further delivering
$$J_{4}(T)\ll T^{7/8}.$$ By similar reasons and on recalling (\ref{J1J9}) one has
$$J_{1}(T)=T\sum_{m^{l_{1}}\leq \sqrt{aT_{1}}}m^{-(l_{1}+l_{2}+l_{3})/2}+O(T^{7/8})=\zeta\big((l_{1}+l_{2}+l_{3})/2\big)T+O(T^{7/8}),$$ whence a combination of the preceding asymptotic relation and Lemma \ref{prop4789} completes the proof of the proposition at hand.

\end{proof}

\section{Bounds for integrals of unimodular functions}\label{simple}
The following lines will be devoted to provide estimates for some of the integrals involved in (\ref{piresiti}) via a straightforward application of Titchmarsh \cite[Lemmata 4.2, 4.4]{Tit}. The results in this section shall be obtained for positive real coefficients $a,b,c>0$ and arbitrary fixed $0<\theta<1$ for future use in later work. It then seems appropiate to define for pairs of positive real numbers $r,s>0$ the integral
$$Y_{2,r,s}^{\theta}(T)=\int_{0}^{T}D_{\theta}(1/2+ait)D(1/2+irt)D(1/2-ist)\chi(1/2-irt)dt.$$ 
Equipped with these definitions we consider \begin{equation}\label{I2rs}I_{2,\theta}(T)=Y_{2,b,c}^{\theta}(T)+Y_{2,c,b}^{\theta}(T).\end{equation}
Likewise, we further define the pairs of functions $$f_{3,\theta}(t)=D_{\theta}(1/2+ait)\ \ \ \ \ \ \ \ \text{and}\ \ \ \ \ \ \ \ \ f_{4}(t)=\chi(1/2+ait)D(1/2-ait),$$ and the integral
$$I_{3,\theta}(T)=\int_{0}^{T}D(1/2+bit)D(1/2+cit)\chi(1/2-bit)\chi(1/2-cit)f_{3,\theta}(t)dt,$$ the integral $I_{4}(T)$ being defined by the above formula with $f_{4}(t)$ replacing $f_{3,\theta}(t)$. 
\begin{lem}\label{lem604}
Let $a,r,s>0$ be real numbers with the property that $\min(r,s)> 2\theta a$ or $\min(r,s)=a$ and $\theta=1/2$. Then 
$$Y_{2,r,s}^{\theta}(T)\ll T^{1/2}Q^{1/2}.$$ Moreover, whenever $a\leq c\leq b$ one has
$$\max\big(I_{2,\theta}(T),I_{3,\theta}(T)\big)\ll T^{1/2}Q^{1/2},\ \ \ \ \ \ \ \ \ \ \ \ \ I_{4}(T)\ll T^{3/4}.$$ 
\end{lem}
\begin{proof}
We begin our proof by analysing first $Y_{2,r,s}^{\theta}(T)$, it being required beforehand to note that it is a consequence of Montgomery and Vaughan \cite[Theorem 2]{MonVau2} that
\begin{equation}\label{rrio}\int_{0}^{T}\lvert D_{\theta}(1/2+ait)\rvert^{2}dt\ll T\log T.\end{equation} We then use the approximation formula for $\chi(1/2-rit)$ contained in Lemma \ref{lem601} to obtain
\begin{equation}\label{lbr}Y_{2,r,s}^{\theta}(T)=e^{-i\pi/4}\sum_{\bfn\in \mathcal{B}_{a,r,s,\theta}}P_{\bfn}^{-1/2}\int_{N_{\bfn,\theta}}^{T}e^{iF_{2}(t)}dt+O(T^{\varepsilon}),\end{equation} where the function $F_{2}(t)$ is defined by means of the formula
$$F_{2}(t)=rt\log rt-rt(\log 2\pi+1)-t\log (n_{1}^{a}n_{2}^{r}/n_{3}^{s}),$$ and we employed equation (\ref{rrio}) in conjunction with Holder's inequality to deduce $$\int_{0}^{t}\lvert D_{\theta}(1/2+aiw)D(1/2+irw)D(1/2-isw)\lvert dw\ll t^{1+\varepsilon},$$ such an observation combined with integration by parts and the aforementioned approximation formula delivering the desired error term. It also seems appropiate to differentiate the line cognate to the definition of $F_{2}(t)$ and recall (\ref{n1prima}) to obtain
$$F_{2}'(t)=r\log t-\theta a\log a+\frac{r}{2}\log r-r\log 2\pi-\log(n_{1,\theta}'^{a}n_{2}'^{r})+s\log n_{3}.$$ 

The reader may find it useful to observe that then $F_{2}'(t)$ is an increasing function. In view of the fact that $r\geq a$, it then transpires that \begin{align*}\label{aka}F_{2}'(N_{\bfn,\theta})\geq &\big(r-2\theta a\big)\log\big(\max(n_{1,\theta}'^{1/2\theta},n_{2}',n_{3}')\big)+\frac{r}{2}\log r-\theta a\log a+s\log n_{3},\end{align*} whence by monotonicity the same holds in the interval $[N_{\bfn,\theta},T]$. In view of the above considerations, it transpires that whenever $r>2\theta a$ then for triples with one of the components being large enough, an application of Titchmarsh \cite[Lemma 4.2]{Tit} suffices to deduce that the cognate integral is $O(1)$. Moreover, if the triples $(n_{1},n_{2},n_{3})$ are bounded by a fixed constant then the corresponding contribution to $Y_{2,r,s}^{\theta}(T)$ would then be $O(T^{1/2})$ by Titchmarsh \cite[Lemma 4.4]{Tit}. The preceding discussion yields
\begin{equation*}Y_{2,r,s}^{\theta}(T)\ll T^{1/2}+\sum_{\bfn\in \mathcal{B}_{a,r,s,\theta}}P_{\bfn}^{-1/2}\ll T^{1/2}Q^{1/2}.\end{equation*} 

If on the contrary $r=a$ and $\theta=1/2$ then it is apparent that
$$F_{2}'(N_{\bfn})\geq a\log\big(\max(n_{1},n_{2})/\min(n_{1},n_{2})\big)+s\log n_{3},$$ the above expression entailing $F_{2}'(t)\gg 1$ in the interval at hand if $n_{3}\geq 2$ and the corresponding contribution being $O(T^{3/4})$. If on the contrary $n_{3}=1$ then an application of Titchmarsh \cite[Lemmata 4.2, 4.4]{Tit} enables one to deduce that the corresponding contribution is bounded above by a constant times
\begin{align*}T^{1/2}\sum_{n_{1}\leq \sqrt{aT_{1}}}n_{1}^{-1}+\sum_{\substack{n_{1},n_{2}\leq \sqrt{aT_{1}}\\ n_{1}>n_{2}}}\frac{n_{1}^{-1/2}n_{2}^{-1/2}}{\lvert\log(n_{1}/n_{2})\rvert}\ll T^{1/2}\log T+\sum_{n_{1},r\leq \sqrt{aT_{1}}}\frac{1}{r}\ll T^{1/2}\log T,
\end{align*}
the first summand in the above equations encompassing the instance when $n_{1}=n_{2}.$ The preceding discussion then yields
$$Y_{2,r,s}^{\theta}(T)\ll T^{1/2}Q^{1/2},$$
which delivers the first part of the statement. It might be worth noting that recalling (\ref{I2rs}) and applying the estimates for $Y_{2,b,c}^{\theta}(T)$ and $Y_{2,c,b}^{\theta}(T)$ obtained herein one gets $$I_{2,\theta}(T)\ll T^{1/2}Q^{1/2}.$$ 

We employ for the perusal of $I_{3,\theta}(T)$ the approximation formula in Lemma \ref{lem601} and the argument after (\ref{lbr}), it leading to the error term thereof, to obtain
$$I_{3,\theta}(T)=-i\sum_{\bfn\in \mathcal{B}_{a,b,c,\theta}}P_{\bfn}^{-1/2}\int_{N_{\bfn,\theta}}^{T}e^{iF_{3}(t)}dt+O(T^{\varepsilon}),$$where the derivative of the function $F_{3}(t)$ is 
$$F_{3}'(t)=(b+c)\log t+b\log b+c\log c-(b+c)\log2\pi-\log(n_{1}^{a}n_{2}^{b}n_{3}^{c}),$$ it being desirable to avoid giving account of the definition of $F_{3}(t)$ for the sake of concission. We observe that then $F_{3}'(t)$ is monotonic and that $$F_{3}'(N_{\bfn,\theta})\geq (b+c-2\theta a)\log\big(\max(n_{1,\theta}'^{1/2\theta},n_{2}',n_{3}')\big)+\frac{b}{2}\log b+\frac{c}{2}\log c-\theta a\log a$$ in the interval of integration at hand, wherein the reader may find it useful to recall (\ref{n1prima}), whence in a similar fashion as above, Titchmarsh \cite[Lemmata 4.2, 4.4]{Tit} yields
\begin{equation*}I_{3,\theta}(T)\ll T^{1/2}Q^{1/2}.\end{equation*}

In order to estimate $I_{4}(T)$ we use as customary the formulae in Lemma \ref{lem601} to obtain
$$I_{4}(T)=e^{-i\pi/4}\sum_{\bfn\in \mathcal{B}_{a,b,c}}P_{\bfn}^{-1/2}\int_{N_{\bfn}}^{T}e^{iF_{4}(t)}dt+O(T^{\varepsilon}),$$it being convenient for the sake of brevity to avoid providing a definition for the above function and indicate that its derivative is 
$$F_{4}'(t)=(b+c-a)\log t+b\log b+c\log c-a\log a-(b+c-a)\log2\pi-\log (n_{2}^{b}n_{3}^{c}/n_{1}^{a}).$$ We observe that then $F_{4}'(t)$ is monotonic and $$F_{4}'(N_{\bfn})\geq (b+c-2a)\log \big(\max(n_{1}',n_{2}',n_{3}')\big)+a\log n_{1}'+\frac{b}{2}\log b+\frac{c}{2}\log c-\frac{a}{2}\log a,$$ whence Titchmarsh \cite[Lemmata 4.2, 4.4]{Tit} in conjunction with the arguments previously employed then yields
\begin{equation*}I_{4}(T)\ll T^{3/4},\end{equation*}
as desired.\end{proof}

\section{An application of the stationary phase method}\label{sec605}

By proceeding in a routinary manner we shall ascend to a position from which an application of Titchmarsh \cite[Lemmata 4.2, 4.4]{Tit} will already suffice to obtain unconditional bounds for $I_{5}(T)$ of the required precision. Nonetheless, an approximation superior to that obtained in the unconditional analysis can be pursued via the stationary phase method if one further assumes Conjecture \ref{conj11}. We consider
$$I_{5}(T)=\int_{0}^{T}D(1/2-ait)D(1/2-bit)D(1/2-cit)\chi(1/2+ait)dt,$$ and gather an unconditional formula concerning such a term in the following lemma.
\begin{lem}\label{lem6054}
Let $a\leq c\leq b$ with $a,b,c\in\mathbb{R}_{+}$ and let $I_{5}(T)$ be defined as above. Then,
\begin{align*}I_{5}(T)=2\pi a^{-1}\sum_{\substack{N_{\bfn}\leq c_{\bfn}\leq T\\\bfn\in \mathcal{B}_{a,b,c}}}n_{2}^{b/2a-1/2}n_{3}^{c/2a-1/2}e( n_{1}n_{2}^{b/a}n_{3}^{c/a})+R(T),
\end{align*}wherein the term $R(T)$ satisfies
$$R(T)\ll T^{3/4}(\log T)+T^{1/2+a/2c}(\log T)^{2}.$$ Moreover, when $a< b$ then upon recalling $\eta$ defined after (\ref{jodo}) one has
$$I_{5}(T)\ll T^{3/4+a/4c}(\log T)^{\eta}.$$
\end{lem}
\begin{proof}
We begin the discussion by recalling (\ref{NNN}) to the reader and employing as customary Lemma \ref{lem601} to approximate $\chi(1/2+ait)$ and express the above term as 
\begin{equation}\label{I4}I_{5}(T)=e^{i\pi/4}\sum_{\bfn\in \mathcal{B}_{a,b,c}}P_{\bfn}^{-1/2}\int_{N_{\bfn}}^{T}e^{iF_{5}(t)}dt+O(T^{\varepsilon}),\end{equation} where the function $F_{5}(t)$ is defined by the relation
$$F_{5}(t)=-at\log at+at(\log 2\pi+1)+t\log (n_{1}^{a}n_{2}^{b}n_{3}^{c}),$$ the customary argument utilised in (\ref{lbr}) playing a prominent role to obtain the above error term. We compute its derivative
\begin{equation}\label{F44}F_{5}'(t)=-a\log t-a\log a+a\log 2\pi+\log(n_{1}^{a}n_{2}^{b}n_{3}^{c}),\end{equation} whence on denoting $c_{\bfn}=2\pi n_{1}n_{2}^{b/a}n_{3}^{c/a}/a$ it transpires that $F_{5}'(c_{\bfn})=0.$ Before making further progress it seems desirable to define the functions $A_{1}(\bfn)=N_{\bfn}$ and $A_{2}(\bfn)=T.$ We combine both the application of Titchmarsh \cite[Lemmata 4.2, 4.4]{Tit} with Graham and Kolesnik \cite[Lemma 3.4]{Gra} to the integral in (\ref{I4}) to obtain
\begin{align}\label{S1}I_{5}(T)=&
2\pi a^{-1}\sum_{\substack{N_{\bfn}\leq c_{\bfn}\leq T\\ \bfn\in \mathcal{B}_{a,b,c}}}n_{2}^{b/2a-1/2}n_{3}^{c/2a-1/2}e( n_{1}n_{2}^{b/a}n_{3}^{c/a})+O\big(T^{3/4}\big)
\\
&+O\big(E_{1}(T)+E_{2}(T)\big),\nonumber
\end{align} the above error $O(T^{3/4})$ arising after summing over tuples $\bfn\in \mathcal{B}_{a,b,c}$ the error term $R_{2}$ in the statement of Graham and Kolesnik \cite[Lemma 3.4]{Gra}, it being on this setting $O(1)$, stemming after the application of the stationary phase method to each of the integrals, and wherein the above terms are defined by means of the formula \begin{equation}\label{EjE}E_{j}(T)=\sum_{\substack{A_{j}(\bfn)/2\leq c_{\bfn}\leq 2A_{j}(\bfn)\\ \bfn\in \mathcal{B}_{a,b,c}}}P_{\bfn}^{-1/2}\min\big(\lvert F_{5}'(A_{j}(\bfn))\rvert^{-1},A_{j}(\bfn)^{1/2}\big)\ \ \ \ \ \ \ \ \ \ j=1,2.\end{equation}

The reader shall rest assured that further details about such an application will be delivered promptly. It may first be useful to observe that in the preceding lines we implicitly applied Graham and Kolesnik \cite[Lemma 3.4]{Gra} for the range $2N_{\bfn}\leq c_{\bfn}\leq T/2$ to the integral
$$\int_{c_{\bfn}/2}^{2c_{\bfn}}e^{iF_{5}(t)}dt$$ and estimated the remaining parts of the integral in (\ref{I4}) employing Titchmarsh \cite[Lemma 4.2]{Tit}. Observe that then the error term arising from such remaining parts is $O(1)$. Likewise, if $N_{\bfn}\leq c_{\bfn}< 2N_{\bfn}$ and $c_{\bfn}\leq T/2$ then Graham and Kolesnik \cite[Lemma 3.4]{Gra} is applied with the choices $\alpha=N_{\bfn}$ and $\beta=2c_{\bfn}$, the error term $E_{1}(T)$ in the above formula arising after such an application. If instead $T/2< c_{\bfn}\leq T$ and $2N_{\bfn}\leq c_{\bfn}$ then the latter lemma shall be employed by taking $\alpha=c_{\bfn}/2$ and $\beta=T$, the error term $E_{2}(T)$ in the preceding equation stemming from this instance. If on the contrary one has $4N_{\bfn}>T$ then $\alpha=N_{\bfn}$ and $\beta=T$ will suffice to obtain the desired result. Finally, whenever either $N_{\bfn}/2\leq c_{\bfn}< N_{\bfn}$ or $T< c_{\bfn}\leq 2T$ then Titchmarsh \cite[Lemmata 4.2, 4.4]{Tit} will provide a contribution which will be absorbed in the error term of the above equation.

We resume our discussion by denoting first $E_{1,i}(T)$ to the contribution to $E_{1}(T)$ of tuples satifying $n_{i}'=\max(n_{1}',n_{2}',n_{3}')$ respectively for each $1\leq i\leq 3,$ it being appropiate to begin by analysing $E_{1,1}(T)$. To this end we define, for each $(n_{2},n_{3})$, the parameter \begin{equation}\label{N1N1}U_{1}= n_{2}^{b/a}n_{3}^{c/a}.\end{equation} We observe for further use that employing this notation then the range of summation $N_{\bfn}/2\leq c_{\bfn}\leq 2N_{\bfn}$ in the first error term of (\ref{S1}) is equivalent to $n_{1}/2\leq U_{1}\leq 2n_{1}$, and that on recalling (\ref{F44}) one has that
$$F_{5}'(N_{\bfn})=a\log U_{1}-a\log n_{1}.$$ 
We trivially bound the minimum cognate to the formula (\ref{EjE}) by $N_{\bfn}^{1/2}$ for the instance $\lvert n_{1}-U_{1}\rvert\leq 1$ and apply the same argument as the one deployed in (\ref{N1aa}) and the subsequent equations for the instance $U_{1}/2\leq n_{1}\leq 2U_{1}$ to deduce
\begin{align*}
\sum_{n_{2},n_{3}\ll \sqrt{T}}&
\sum_{\substack{U_{1}/2\leq n_{1}\leq 2U_{1}\\ n_{1}\ll \sqrt{T}}}P_{\bfn}^{-1/2}\min\big(\lvert F_{5}'(N_{\bfn})\rvert^{-1},N_{\bfn}^{1/2}\big)\nonumber
\\
&\ll (\log T)\sum_{U_{1}\ll\sqrt{T}}n_{2}^{-1/2}n_{3}^{-1/2}U_{1}^{1/2}\ll T^{1/4}(\log T)\sum_{\substack{n_{2}\ll \sqrt{T}\\ n_{3}\ll \sqrt{T}}}n_{2}^{-1/2}n_{3}^{-1/2}\ll T^{3/4}(\log T).
\end{align*} 
The reader may note that for $n_{1}$ outside of the range considered above then $\lvert F_{5}'(N_{\bfn})\rvert^{-1}\ll 1$, whence the contribution to the above sum arising from such tuples  is $O(T^{3/4})$, and hence
$$E_{1,1}(T)\ll T^{3/4}\log T.$$ 

We next focus on the term $E_{1,2}(T).$ It seems appropiate to observe first that when $b\geq 2a$ then in view of (\ref{F44}) one has $$ F_{5}'(N_{\bfn})=\log (n_{1}^{a}n_{2}^{b-2a}n_{3}^{c})+a\log (b/a).$$ It therefore transpires that whenever either $n_{1}$ or $n_{3}$ are sufficiently large then one may estimate such a contribution to $E_{1,2}(T)$ by $O(T^{3/4}),$ the instance when both of the entries are bounded being estimated by means of the trivial observation $N_{\bfn}^{1/2}\ll T^{1/2},$ which in turn yields the bound $O(T^{3/4})$ for such a contribution. Suffices it then to consider the case $b<2a.$ To this end it seems worth defining for each $(n_{1},n_{3})$ the parameter $$N_{2}=\big(b^{a}n_{1}^{a}n_{3}^{c}/a^{a}\big)^{1/(2a-b)}$$ and observe that in this particular instance it is apparent that 
$$F_{5}'(N_{\bfn})=(2a-b)\big(\log N_{2}-\log n_{2}\big).$$
Therefore, an analogous argument to the one utilised above enables one to deduce
$$E_{1,2}(T)\ll T^{3/4}\log T.$$ The term $E_{1,3}(T)$ presents an analogous behaviour to that of $E_{1,2}(T)$, a similar analysis within this circle of ideas thus delivering $E_{1,3}(T)=O(T^{3/4}\log T),$ and hence
\begin{equation}\label{EEE}E_{1}(T)\ll T^{3/4}\log T.\end{equation} 
 
We finally analyse the error term $E_{2}(T)$ and begin by noting for convenience that whenever $T/2\leq c_{\bfn}\leq 2T$ then $T^{1/2}\ll n_{2}^{b/a}n_{3}^{c/a}\ll T.$ We take $\Lambda_{1}= aT_{1}n_{2}^{-b/a}n_{3}^{-c/a}$ and observe on recalling (\ref{F44}) that 
$$F_{5}'(T)=a(\log n_{1}-\log \Lambda_{1}).$$ It also seems worth noting in view of the above considerations that $n_{1}\asymp \Lambda_{1}$, such a relation in conjunction with the underlying restriction $n_{1}\leq \sqrt{aT}$ in turn entailing the bounds
$$\Lambda_{1}\ll T^{1/2},\ \ \ \ \ \ \ \ \Lambda_{1}^{1/2}\ll \Lambda_{1}^{-1/2}T^{1/2}.$$ The same arguments utilised on previous occasions then deliver
\begin{align}\label{E2T}
E_{2}(T)&
\ll T^{1/2}(\log T)\sum_{n_{2}^{b}n_{3}^{c}\ll T^{a}}n_{2}^{-1/2}n_{3}^{-1/2}\Lambda_{1}^{-1/2}\ll (\log T)\sum_{n_{2}^{b}n_{3}^{c}\ll T^{a}}n_{2}^{-1/2+b/2a}n_{3}^{-1/2+c/2a}\nonumber
\\
&\ll (\log T)T^{1/2+a/2b}\sum_{n_{3}\ll T^{a/c}}n_{3}^{-1/2-c/2b}\ll (\log T)^{2}\min(T^{1/2+a/2c},T^{3/4+(2a-c)/4b}),
\end{align}
where in the above sums we omitted writing the restriction $n_{3}\ll \sqrt{T}$. The preceding discussion then yields the first statement of the lemma. 

In order to conclude the proof it thus remains to focus our attention on the main term of (\ref{S1}), which we denote by $P(T)$. By bounding the exponential sum on $P(T)$ trivially we then find that
\begin{align*}P(T)&
\ll \sum_{\substack{n_{2}^{b/a}n_{3}^{c/a}\ll T\\ n_{3}\leq \sqrt{cT_{1}}}}n_{2}^{b/2a-1/2}n_{3}^{c/2a-1/2}\min\big(T^{1/2},Tn_{2}^{-b/a}n_{3}^{-c/a}\big)\ll P_{1}(T)+P_{2}(T),\end{align*}
where the terms in the above line are defined by means of  \begin{equation*}P_{1}(T)= T^{1/2}\sum_{n_{2}^{b/a}n_{3}^{c/a}\leq \sqrt{cT_{1}}}n_{2}^{b/2a-1/2}n_{3}^{c/2a-1/2},\end{equation*} and $$P_{2}(T)=T\sum_{\substack{n_{2}^{b/a}n_{3}^{c/a}> \sqrt{cT_{1}}\\n_{3}\leq \sqrt{cT_{1}}}}n_{2}^{-1/2-b/2a}n_{3}^{-1/2-c/2a}.$$ We shall begin our investigation examining first $P_{1}(T)$. Summing over $n_{3}$ we obtain
\begin{equation*}P_{1}(T)\ll T^{3/4+a/4c}\sum_{n_{2}\leq (cT_{1})^{a/2b}}n_{2}^{-1/2-b/2c}\ll T^{3/4+a/4c}(\log T)^{\eta}.\end{equation*} Likewise, we have
\begin{equation*}P_{2}(T)\ll T\sum_{\substack{n_{3}^{c/a}> \sqrt{cT_{1}}\\ n_{3}\leq \sqrt{cT_{1}}}}n_{3}^{-1/2-c/2a}+T^{3/4+a/4b}\sum_{n_{3}^{c/a}\leq \sqrt{cT_{1}}}n_{3}^{-1/2-c/2b}\ll T^{3/4+a/4c}(\log T)^{\eta}.\end{equation*}

It is worth noting that whenever $a=c$ then one always has $n_{3}\leq \sqrt{cT_{1}}$, whence in this particular instance there is no first summand on the right side of the above equation. The preceding discussion in conjunction with (\ref{S1}), (\ref{EEE}) and (\ref{E2T})  and the observation that $(2a-c)c<ab$ under the proviso $a<b$ then delivers the estimate
$$I_{5}(T)\ll T^{3/4+a/4c}(\log T)^{\eta},$$ as desired.
\end{proof}

The argument to bound $I_{5}(T)$ could have been employed after a straight application of Titchmarsh \cite[Lemmata 4.2, 4.4]{Tit}. We invoked Graham and Kolesnik \cite[Lemma 3.4]{Gra} herein because one may derive an asymptotic formula comprising main terms which shall eventually cancel with analogous terms stemming from the computation of $I_{1}(T)$ in the context underlying Theorem \ref{thm601}, it therefore being desirable to have a succinct discussion concerning the main term $P(T)$ in (\ref{S1}). We recall (\ref{0.1212}) and (\ref{N1N1}) and denote by $M_{1}(T)$ to the contribution to $P(T)$ of tuples with the property that $U_{1}$ is not an integer. Likewise, let $K_{1}(T)$ be the contribution of tuples for which $U_{1}$ is a natural number.
\begin{lem}\label{lem605d}
Let $a,b,c\in\mathbb{R}_{+}$ such that $a< c\leq b$. Then one has
\begin{align*}\label{riips}I_{5}(T)=J_{3}(T)+J_{4}(T)+M_{1}(T)+O\big(T^{3/4}(\log T)+T^{1/2+a/2c}(\log T)^{2}+T^{5/4-c/4a}\big).
\end{align*} 
\end{lem}
\begin{proof}
We observe first that an application of the preceding lemma enables one to deduce the formula
\begin{equation}\label{TT5} I_{5}(T)=K_{1}(T)+M_{1}(T)+O\big(T^{1/2+a/2c}(\log T)^{2}+T^{3/4}\log T\big).\end{equation} It seems worth noting in view of the restrictions underlying the sum in (\ref{S1}) that the cognate triples satisfy the inequality $\max(M_{2},M_{3})\leq n_{1}$, wherein
$$M_{2}=ab^{-1}n_{2}^{2-b/a}n_{3}^{-c/a},\ \ \ \ \ \ \ \ \ M_{3}=ac^{-1}n_{3}^{2-c/a}n_{2}^{-b/a},$$ it being non-trivial only when $c<2a.$ Therefore, it transpires upon defining the sum \begin{equation*}L(T)=\sum_{n_{2}^{b/a}n_{3}^{c/a}\leq aT_{1}}\max(M_{2},M_{3})n_{2}^{b/2a-1/2}n_{3}^{c/2a-1/2}\ \ \ \ \ \ \ \ \ \  \ \ \end{equation*} that then 
\begin{equation}\label{K1L}K_{1}(T)=K_{1}'(T)+O(L(T)),\end{equation} wherein the above equation the term $K_{1}'(T)$ denotes the sum in (\ref{S1}) over tuples without the aforementioned restriction on $n_{1}$. We observe that then
\begin{align*}L(T)&\ll \sum_{n_{2},n_{3}\ll T^{1/2}}n_{2}^{3/2-b/2a}n_{3}^{-c/2a-1/2}+\sum_{n_{2},n_{3}\ll T^{1/2}}n_{2}^{-b/2a-1/2}n_{3}^{3/2-c/2a}\ll 1+T^{5/4-c/4a}.\end{align*}

We shift our attention to the analysis of the main term in (\ref{K1L}) and write for convenience \begin{equation*}K_{1}'(T)=K_{1,1}(T)+K_{1,2}(T),\end{equation*} where $K_{1,1}(T)$ is defined upon recalling (\ref{N1N1}) as the corresponding sum over the tuples satisfying $U_{1}\leq \sqrt{aT_{1}}$, the term $K_{1,2}(T)$ denoting the sum over tuples with the property that $U_{1}>\sqrt{aT_{1}}$. Therefore, one then has 
\begin{equation}\label{K11}K_{1,1}(T)=J_{4}(T)+O\big(R_{a,b,c}(T)\big),\end{equation} wherein
\begin{equation*}R_{a,b,c}(T)=\sum_{\substack{U_{1}^{a}=n_{2}^{b}n_{3}^{c}\\ U_{1}\leq aT_{1}}}U_{1}^{1/2}n_{2}^{-1/2}n_{3}^{-1/2},\end{equation*} the term $J_{4}(T)$ having been defined in (\ref{0.1212}). It is then apparent that 
$$R_{a,b,c}(T)\ll \sum_{\substack{n_{2}^{b/a}n_{3}^{c/a}\ll T}}n_{2}^{b/2a-1/2}n_{3}^{c/2a-1/2},$$whence by summing first over $n_{3}$ one has that
$$R_{a,b,c}(T)\ll T^{1/2+a/2c}\sum_{\substack{n_{2}\ll T^{1/2}}}n_{2}^{-1/2-b/2c}\ll T^{1/2+a/2c}(\log T).$$ Combining the preceding discussion with equation (\ref{K11}) then delivers  \begin{equation*}K_{1,1}(T)=J_{4}(T)+O\big(T^{1/2+a/2c}(\log T)\big).\end{equation*} We recall the term $J_{3}(T)$, it being introduced in (\ref{0.1212}) to the end of noting in a similar manner that one then has 
\begin{equation*}K_{1,2}(T)= J_{3}(T)+O\big(T^{1/2+a/2c}(\log T)\big),\end{equation*} a combination of the above equations in conjunction with (\ref{TT5}) yielding the desired result.
\end{proof}

The following lemma shall utilise the above discussion to sharpen the estimate obtained in Lemma \ref{lem6054} conditionally on the validity of Conjecture \ref{conj11}.

\begin{lem}\label{lem605}
Let $a,b,c\in\mathbb{N}$ with the property that $a< c\leq b$. Assuming Conjecture \ref{conj11} one has
$$M_{1}(T)\ll T^{1/2+a/(a+c)+\varepsilon},$$ whence it transpires that
\begin{align*}I_{5}(T)=J_{3}(T)+J_{4}(T)+O\big(T^{3/4}(\log T)+T^{1/2+a/(a+c)+\varepsilon}\big).
\end{align*}
\end{lem}
\begin{proof}
We recall (\ref{N1N1}) and the definitions right after the end of the proof of Lemma \ref{lem6054} for the purpose of noting first that the sum over $n_{1}$ in the definition for $M_{1}(T)$ is a sum of a geometric progression, whence
$$M_{1}(T)\ll \sum_{n_{2}^{b/a}n_{3}^{c/a}\ll T}n_{2}^{b/2a-1/2}n_{3}^{c/2a-1/2}\min\big(\norm{U_{1}}^{-1},T/U_{1}\big),$$
wherein the above sum we omitted writing for the sake of concision that $n_{2}^{b/a}n_{3}^{c/a}$ is not an integer. In order to progress in the proof we recall (\ref{Nbc}) and denote for further convenience $D= n_{2}^{b}n_{3}^{c}-N_{1}^{a}\neq 0,$ the last property stemming from the above assumption. Then one may apply the mean value theorem to obtain 
$$\norm{n_{2}^{b/a}n_{3}^{c/a}}=\lvert (N_{1}^{a}+D)^{1/a}-N_{1}\rvert\asymp \lvert D\rvert N_{1}^{1-a},$$ whence the inequality $$\lvert D\rvert\gg N_{1}^{a-1-\varepsilon}(n_{2}n_{3})^{-1},$$ it in turn being a consequence of Conjecture \ref{conj11}, in conjunction with the above bound, yields the estimate
\begin{equation*}\norm{U_{1}}^{-1}\ll (n_{2}n_{3})^{1+\varepsilon}.\end{equation*} By the preceding discussion one gets
$$M_{1}(T)\ll \sum_{n_{2}^{b/a}n_{3}^{c/a}\ll T}n_{2}^{b/2a-1/2}n_{3}^{c/2a-1/2}\min\big((n_{2}n_{3})^{1+\varepsilon},T/U_{1}\big).$$

We shall divide the sum into parts for convenience. We denote by $M_{1,1}(T)$ to the contribution to $M_{1}(T)$ of tuples satisfying $n_{2}^{a+b}n_{3}^{a+c}\leq T^{a}.$ Then one has that
\begin{align*}M_{1,1}(T)&\ll T^{\varepsilon}\sum_{\substack{n_{2}^{a+b}n_{3}^{a+c}\leq T^{a}}}n_{2}^{1/2+b/2a}n_{3}^{1/2+c/2a}
\\
&\ll T^{(b+3a)/2(a+b)+\varepsilon}\sum_{n_{3}\leq T^{a/(a+c)}}n_{3}^{-(a+c)/(a+b)}\ll T^{1/2+a/(a+c)+\varepsilon}.\end{align*} Likewise, let $M_{1,2}(T)$ denote the contribution to $M_{1}(T)$ of tuples with the property that $T^{a}<n_{2}^{a+b}n_{3}^{a+c}.$ Then one readily sees that
\begin{align*}M_{1,2}(T)&\ll
T\sum_{\substack{T^{a}<n_{2}^{a+b}n_{3}^{a+c}}}n_{2}^{-1/2-b/2a}n_{3}^{-1/2-c/2a}
\\
&\ll T^{1+(a-b)/2(a+b)}\sum_{n_{3}^{a+c}\leq T^{a}}n_{3}^{-(a+c)/(a+b)}+T\sum_{T^{a}< n_{3}^{a+c}}n_{3}^{-1/2-c/2a} \ll T^{1/2+a/(a+c)},\end{align*} the combination of the preceding estimates thereby delivering
$$M_{1}(T)\ll T^{1/2+a/(a+c)+\varepsilon}.$$

We shall conclude the proof of the lemma by inserting the above estimates in the equation on the statement of Lemma \ref{lem605d} in conjunction with the observation that \begin{align}\label{prifs}5/4-c/4a&=\frac{(5a-c)(a+c)}{4(a+c)a}=\frac{(5a-c)(a+c)-6a^{2}-2ac}{4(a+c)a}+1/2+\frac{a}{a+c}\nonumber
\\
&=\frac{-(a-c)^{2}}{4(a+c)a}+1/2+\frac{a}{a+c}<1/2+\frac{a}{a+c}.
\end{align}

\end{proof}

The reader shall rest assured that there is no substantial cancellation arising from the sum of both $M_{1}(T)$ and $J_{2,2}(T)$ defined in (\ref{J22}). The following lemma will be devoted to unconditionally sharpen the above analysis for the case $a=1$, it being convenient to recall (\ref{0.1212}) to such an end, for the purpose of improving the corresponding error terms and giving account of secondary terms when $a=c=1$.
\begin{lem}\label{lemr}
Let $a=1$ and $b,c\in\mathbb{N}$ having the property that $b\geq c> 1$. Then one has that
\begin{align*}I_{5}(T)=J_{3}(T)+J_{4}(T)+O\big(T^{3/4}(\log T)^{2}\big).
\end{align*}
\end{lem}
\begin{proof}
Recalling the term $M_{1}(T)$ defined above the statement of Lemma \ref{lem605d} and equation (\ref{S1}) we observe that whenever $a=1$ then $M_{1}(T)=0$. The lemma follows applying Lemma \ref{lem605d} and noting that the ensuing error terms in this context are  $O\big(T^{3/4}(\log T)^{2}\big).$
\end{proof}

The remaining discussion in this section shall be devoted to refine the error term stemming from Lemma \ref{lem6054} for the instance $a=c=1$, it having been thought pertinent to defer the explicit computation of the corresponding main terms to a later point in the memoir.
\begin{lem}\label{lemot}
Let $a=c=1$ and $b\in\mathbb{N}$. Then one has that
$$I_{5}(T)=K_{1}(T)+O\big(T^{3/4}\log T\big)$$ when $b\neq 2$, the term $K_{1}(T)$ having been defined right after the statement of Lemma \ref{lem605} and amounting in this context to
$$K_{1}(T)=2\pi\sum_{N_{\bfn}\leq c_{\bfn}\leq T}n_{2}^{b/2-1/2}n_{3}^{c/2-1/2}.$$ If $b=2$ then the same formula replacing the above error term by $O(T^{3/4+\varepsilon})$ holds.
\end{lem}

\begin{proof}
We shall focus first on the case $b>1$. Equation (\ref{S1}) in conjunction with (\ref{EEE}) enables one to reduce the problem to that of bounding $E_{2}(T)$ appropiately. To such an end, we introduce for each integer $n_{2}\in\mathbb{N}$ the parameter $\Lambda_{3}=T_{1}n_{2}^{-b}$ and observe that upon recalling (\ref{F44}) it transpires that then $$F_{5}'(T)=\log n_{1}n_{3}-\log \Lambda_{3}.$$We then note that whenever $T/2\leq c_{\bfn}\leq 2T$ then $n_{1}n_{3}\asymp \Lambda_{3}$. We remind the reader of the definition of (\ref{EjE}) for the purpose of deducing the estimate
\begin{align*}E_{2}(T)&\ll \sum_{T/2\leq c_{\bfn}\leq 2T}P_{\bfn}^{-1/2}\min\big(\lvert \log(n_{1}n_{3}/\Lambda_{3})\rvert^{-1},T^{1/2}\big)
\\
&\ll \sum_{\substack{n\asymp \Lambda_{3}\\ \lvert n-\Lambda_{3}\rvert\geq 1}}\Lambda_{3}^{1/2}\frac{d(n)n_{2}^{-1/2}}{\lvert n-\Lambda_{3}\rvert}+T^{1/2+\varepsilon}\sum_{n_{2}^{b}\ll T}\Lambda_{3}^{-1/2}n_{2}^{-1/2}
\\
&\ll T^{1/2+\varepsilon}\sum_{n_{2}^{b}\ll T}n_{2}^{-b/2-1/2}+T^{\varepsilon}\sum_{n_{2}^{b}\ll T}n_{2}^{b/2-1/2}\ll T^{1/2+1/2b+\varepsilon},
\end{align*}
which completes the proof in the aforementioned context. The lemma for the instance $b=c=1$ follows by utilising routinary arguments to deduce
$$E_{2}(T)\ll \sum_{n\asymp T} d_{3}(n)n^{-1/2}\min\big(\lvert \log(n/T_{1})\rvert^{-1},T^{1/2}\big)\ll T^{1/2+\varepsilon},$$ wherein $d_{3}(n)$ denotes the number of representations of $n$ as a product of three positive integer.
\end{proof}

\section{An intermediate estimate and proof of Theorem \ref{thm601}}\label{sec606}
  
As opposed to the treatment in the previous section, the application of Titchmarsh \cite[Lemmata 4.2 and 4.4]{Tit} shall already suffice to obtain a suitable bound for the term $I_{6}(T)$, it being defined by means of the sum \begin{equation}\label{I66}I_{6}(T)=Y_{6,b,c}(T)+Y_{6,c,b}(T),\end{equation} where for tuples $(r,s)\in\mathbb{R}_{+}^{2}$ the above summands are 
$$Y_{6,r,s}(T)=\int_{0}^{T}D(1/2-ait)D(1/2+rit)D(1/2-sit)\chi(1/2+ait)\chi(1/2-rit)dt.$$ 

\begin{lem}\label{lem606}
Let $(r,s)\in\mathbb{R}_{+}^{2}$ such that $r>a$ and $s\geq a$. Then one has $$Y_{6,r,s}(T)\ll T^{5/4-r/4a}(\log T)^{\tau}+T^{3/4},$$ wherein $\tau=1$ if $s=a$ and $\tau=0$ if $s>a$. In particular, it transpires when $a<c\leq b$ that
$$I_{6}(T)\ll T^{5/4-c/4a}+T^{3/4}.$$
\end{lem}
\begin{proof}
The approximation formulae for $\chi(1/2-rit)$ and $\chi(1/2+ait)$ in Lemma \ref{lem601} in conjunction with the argument following (\ref{lbr}) pertaining to the analysis of the corresponding error term thereof then yield
\begin{equation}\label{I7}Y_{6,r,s}(T)=\sum_{\bfn\in \mathcal{B}_{a,r,s}}P_{\bfn}^{-1/2}\int_{N_{\bfn}}^{T}e^{iF_{6}(t)}dt+O(T^{\varepsilon}),\end{equation}it having been thought convenient for the sake of brevity only to give account of the derivative of the above function, namely
$$F_{6}'(t)=(r-a)\log t+r\log r-a\log a-(r-a)\log2\pi+\log (n_{1}^{a}n_{3}^{s}/n_{2}^{r}).$$ We may discard first the case $N_{\bfn}=2\pi n_{1}^{2}/a$, since then $$F_{6}'(N_{\bfn})\geq (r-a)\log n_{1}'+s\log n_{3}'$$ and a customary application of Titchmarsh \cite[Lemmata 4.2 and 4.4]{Tit} would yield the conclusion that the contribution to (\ref{I7}) corresponding to tuples satisfying such a condition is $O(T^{3/4}).$ 
If instead $N_{\bfn}=2\pi n_{3}^{2}/s$ then $$F_{6}'(N_{\bfn})\geq (r+s-2a)\log n_{3}'+a\log n_{1}',$$ and a routinary application of Titchmarsh \cite[Lemmata 4.2 and 4.4]{Tit} would imply that the contribution to (\ref{I7}) corresponding to tuples satisfying such a condition is $O(T^{3/4}).$

If $N_{\bfn}=2\pi n_{2}^{2}/r$ and $r\geq 2a$ then an analogous argument reveals that $$F_{6}'(N_{\bfn})\geq \log (n_{1}^{a}n_{3}^{s})+(r-2a)\log n_{2}'+\frac{1}{2}r\log r-a\log a,$$ the contribution when both $n_{1},n_{3}$ are bounded or either $n_{1}$ or $n_{3}$ is large enough being $O(T^{3/4})$ by Titchmarsh \cite[Lemmata 4.2, 4.4]{Tit}. We then focus on the instance $r<2a$ and define for convenience the parameter $c_{\bfn}=2\pi \big(a^{a}r^{-r}n_{2}^{r}n_{1}^{-a}n_{3}^{-s}\big)^{1/(r-a)}$, which the reader may check that satisfies $F_{6}'(c_{\bfn})=0$. It seems worth noting that applying Titchmarsh \cite[Lemma 4.4]{Tit} one may deduce that the integral over $[N_{\bfn},T]\cap [c_{\bfn}/2, 2c_{\bfn}]$ is $O(c_{\bfn}^{1/2}).$ Likewise, the integral over the complement of the latter intersection in $[N_{\bfn},T]$ would then be $O(1)$ by Titchmarsh \cite[Lemma 4.2]{Tit}. It may also seem appropiate to remark that were the previous intersection non-empty then one would have $c_{\bfn}\ll T,$ which would in turn imply the inequality
\begin{equation}\label{Ecua1}n_{2}\ll n_{1}^{a/r}n_{3}^{s/r}T^{1-a/r}.\end{equation} We find it desirable to denote by $M_{6}(T)$ to the contribution stemming from the sums over integrals restricted to such intersections. We further divide such a contribution into the one corresponding to tuples $\bfn\in\mathcal{B}_{a,b,c}$ with the property that \begin{equation}\label{Ecua2}n_{1}^{a/r}n_{3}^{s/r}\leq CT^{a/r-1/2}\end{equation} for a suitable constant $C>0$, which will be denoted by $M_{6,1}(T)$, and $M_{6,2}(T)$, that shall denote the contribution stemming from tuples satisfying the converse of (\ref{Ecua2}), the set of which will be referred to by means of the letter $\mathcal{J}_{2}(T)$. For the sake of concision we write $\mathcal{J}_{1}(T)$ for the set of tuples satisfying the inequalities (\ref{Ecua1}) and (\ref{Ecua2}). One then has
\begin{align*}M_{6,1}(T)&
\ll \sum_{\bfn\in\mathcal{J}_{1}(T)}P_{\bfn}^{-1/2}c_{\bfn}^{1/2}\ll \sum_{\bfn\in\mathcal{J}_{1}(T)}n_{1}^{-1/2-a/2(r-a)}n_{2}^{-1/2+r/2(r-a)}n_{3}^{-1/2-s/2(r-a)}
\\
&\ll T^{1-a/2r}\sum_{n_{1}^{a}n_{3}^{s}\ll T^{a-r/2}}n_{1}^{-1/2+a/2r}n_{3}^{-1/2+s/2r}
\\
&\ll T^{1-a/2r+(2a-r)(a+r)/4ar}\sum_{n_{3}\leq \sqrt{cT_{1}}}n_{3}^{-1/2-s/2a}\ll T^{5/4-r/4a}(\log T)^{\tau},
\end{align*}
where $\tau$ was defined in the statement of the lemma. The reader should find it worth noting that in the second line we employed the inequality (\ref{Ecua1}) when summing over $n_{2}$. 

The analysis of $M_{6,2}(T)$, though similar in nature, will depart from the previous procedure in that we will instead utilise the bound $n_{2}\leq \sqrt{bT_{1}}$ in due course. We thus obtain
\begin{align*}
M_{6,2}(T)\ll& \sum_{\bfn\in\mathcal{J}_{2}(T)}P_{\bfn}^{-1/2}c_{\bfn}^{1/2}\ll T^{1/4+r/4(r-a)}\sum_{n_{1}^{a}n_{3}^{s}\gg T^{a-r/2}}n_{1}^{-1/2-a/2(r-a)}n_{3}^{-1/2-s/2(r-a)}.
\end{align*}
It seems appropiate to remark that in the above lines we utilised the fact that the tuples in $\mathcal{J}_{2}(T)$ satisfy (\ref{Ecua1}) and the converse of the inequality (\ref{Ecua2}). Therefore, summing over $n_{1}$ in the second line of inequalities and recalling the assumption $2a>r$ one gets
\begin{align*}M_{6,2}(T)\ll &T^{5/4-r/4a}\sum_{n_{3}\ll T^{(2a-r)/2s}}n_{3}^{-1/2-s/2a}+ T^{1/4+r/4(r-a)}\sum_{n_{3}\gg T^{(2a-r)/2s}}n_{3}^{-1/2-s/2(r-a)}
\\
\ll& T^{5/4-r/4a}(\log T)^{\tau}+T^{3/4+(2a-r)/4s}.
\end{align*}

The reader may find it worth observing that in view of the aforementioned assumption, it transpires that $r-a<a\leq s,$ whence the exponent of $n_{3}$ in the above sum is smaller than $-1$, such a remark justifying the subsequent line of argumentation thereof. We pause our analysis to examine and compare the bounds already obtained, it being worth noting
\begin{align*}3/4+(2a-r)/4s&
=(2a-r)/4s+r/4a-1/2+5/4-r/4a
\\
&=\frac{r(s-a)-2a(s-a)}{4as}+5/4-r/4a\leq 5/4-r/4a,
\end{align*}
where we used the fact that $r<2a$ and $s\geq a.$ The preceding estimates then yield $$\max\big(M_{6,1}(T),M_{6,2}(T)\big)\ll T^{5/4-r/4a}(\log T)^{\tau},$$ as desired. The second statement follows by recalling (\ref{I66}) and applying the result obtained above for $Y_{6,b,c}(T)$ and $Y_{6,c,b}(T)$.
\end{proof}

\emph{Proof of Theorems \ref{thm601}, \ref{prop3kr}, \ref{thm6025} and \ref{thm1.5}.} 
After the prolix discussion held above, in order to complete the proof of Theorem \ref{thm601} it just suffices to combine equation (\ref{osssa}) with Lemmata \ref{lemita601}, \ref{lem604}, \ref{lem6054}, \ref{lem605} and \ref{lem606} and equation (\ref{piresiti}), the proof of Theorem \ref{prop3kr} being in turn a consequence of the same sequel of results in conjunction with an application of Lemma \ref{lemr} instead of Lemmata \ref{lem6054} and \ref{lem605}. It might be worth noting that \begin{align*}5/4-c/4a&
=1/2-c/4a-a/4c+3/4+a/4c
\\
&=-(c-a)^{2}/4ac+3/4+a/4c<3/4+a/4c.
\end{align*}and that
\begin{equation*}3/4+a/4c=\frac{(a-c)(a-2)}{4ac}+1-1/2a+1/2c\leq 1-1/2a+1/2c\end{equation*} when $a\geq 2$ and $a<c$. Therefore, it transpires that the error term $T^{1-1/2a+1/2c}(\log T)^{\eta}$ stemming from Lemma \ref{prop478} dominates over that of $T^{3/4+a/4c}(\log T)^{\eta}$ from Lemma \ref{lem6054} and $T^{5/4-c/4a}$ arising after an application of Lemma \ref{lem606}. Likewise, the reader may find it useful to recall (\ref{prifs}) and observe that $$1/4+3a/4c<1/2+a/(a+c)$$ whenever $a<c$. Consequently, the term $T^{1/2+a/(a+c)+\varepsilon}$ dominates over $T^{1/4+3a/4c+\varepsilon},$ the latter arising after an application of equation (\ref{osssa1}) under the assumption of Conjecture \ref{conj11}. It also seems worth noting for use in future work that a combination of equation (\ref{osss}) with Lemmata \ref{lem604}, \ref{lem605d} and \ref{lem606} delivers 
\begin{equation}\label{superlabel}I_{a,b,c}(T)=\sigma_{a,b,c}T+M_{1}(T)+J_{2,2}(T)+O\big(T^{3/4}(\log T)+T^{1/2+a/2c}(\log T)^{2}+T^{5/4-c/4a}\big).\end{equation}

We conclude our discussion by mentioning that the proof of Theorem \ref{thm6025} departs from that of Theorem \ref{thm601} in the necessity of an application of Proposition \ref{diagono} instead of equation (\ref{osssa}), the rest of the argument being analogous save the absence of the requirement of Lemma \ref{lem605}. Similarly, Theorem \ref{thm1.5} is derived by utilising Proposition \ref{prop55} instead of Proposition \ref{diagono}.

\section{An application of Roth's theorem on diophantine approximation}\label{sec607}
The upcoming discussion shall be devoted to examine the integrals pertaining to $I_{a,b,a}(T)$ that exhibit a different behaviour from those in the setting of Theorem \ref{thm601}. We begin our journey by drawing the reader's attention back to the equation (\ref{piresiti}), wherein the $I_{j}(T)$ were defined at the beginning of each of the above corresponding sections. We recall (\ref{NNN}) and (\ref{Babc}) and write the formula
\begin{equation}\label{I1TTT}I_{1}(T)=S_{a,b}(T)T+J_{2}(T)-J_{4}(T),\end{equation} where the terms $J_{2}(T)$ and $J_{4}(T)$ were defined in (\ref{0.121}) and (\ref{0.1212}) and \begin{equation*}\label{Sab}S_{a,b}(T)=\sum_{\substack{\bfn\in \mathcal{B}_{a,b,a}\\  n_{1}^{a}=n_{2}^{b}n_{3}^{a}}}P_{\bfn}^{-1/2}.\end{equation*}
\begin{lem}
Whenever $a,b\in\mathbb{N}$ with $a<b$ and $(a,b)=1$ one has
$$I_{1}(T)=\frac{\zeta\big((a+b)/2\big)}{2}T\log T+J_{2}(T)+O(T).$$
\end{lem}
\begin{proof}
We begin the discussion by observing that the solutions of the underlying equation in both $S_{a,b}(T)$ and $J_{4}(T)$ can be parametrized by means of the expressions $$n_{1}=m_{3}m_{2}^{b},\ \ \ \ n_{2}=m_{2}^{a},\ \ \ \ n_{3}=m_{3}.$$ Making use of the above remark we obtain
\begin{align*}\label{Sabab}S_{a,b}(T)=&
\sum_{\substack{m_{2}^{b}m_{3}\leq \sqrt{aT_{1}}}}m_{3}^{-1}m_{2}^{-(a+b)/2}=\frac{\log T}{2}\sum_{m_{2}\leq (aT_{1})^{1/2b}}m_{2}^{-(a+b)/2}
\\
&+O(1)-b\sum_{m_{2}\leq (aT_{1})^{1/2b}}(\log m_{2})m_{2}^{-(a+b)/2},\nonumber
\end{align*}
which then yields
\begin{equation*}S_{a.b}(T)=\frac{\zeta\big((a+b)/2\big)}{2}\log T+O(1).\end{equation*}Likewise, we employ the parametrization at hand in the way alluded above to get
\begin{equation*}J_{4}(T)\ll \sum_{m_{3}m_{2}^{b}\leq \sqrt{aT_{1}}}m_{2}^{(3b-a)/2}m_{3}\ll T^{3/4+1/2b-a/4b}\sum_{m_{3}\leq \sqrt{aT_{1}}}m_{3}^{-1/2-1/b+a/2b}\ll T.\end{equation*}
Combining the preceding equations with (\ref{I1TTT}) delivers the desired result.
\end{proof}
The reader may have noticed that one could have refined the above analysis to obtain lower order terms in the asymptotic formula at hand, such an avenue being further explored in the context underlying Theorem \ref{loweuse}. These improvements on this instance though would have been wrought in vain due to the poor understanding of $J_{2}(T)$ that we have.
\begin{lem}\label{rothi}
Let $a,b\in\mathbb{N}$ with $a<b$ as above. Then
$$J_{2}(T)=o(T\log T).$$ Moreover, upon recalling (\ref{QQQQ}) and (\ref{0.121}), if $a=1$ and $b\geq 1$ then for $0<\theta<1$ one has
$$J_{2}^{\theta}(T)\ll T^{1/2}Q^{1/2}.$$
\end{lem}
\begin{proof}
We shall focus first on the instance $a>1$ and $\theta=1/2.$ It seems pertinent to define, for each tuple $(n_{2},n_{3})$ the number $N_{b,a}=[n_{2}^{b/a}n_{3}],$ and on recalling (\ref{J212}), (\ref{J22}) and (\ref{bJ21}) we observe that \begin{equation*}J_{2}(T)= J_{2,2}(T)+O(T^{3/4}).\end{equation*} In order to analyse $J_{2,2}(T)$ we note first that 
\begin{equation}\label{J22ef}J_{2,2}(T)\ll\sum_{\substack{n_{2}\leq \sqrt{bT_{1}}\\ n_{3}\leq \sqrt{aT_{1}}\\ n_{1}=N_{b,a}}}n_{2}^{-1/2-b/2a}n_{3}^{-1}\Big\lvert\int_{N_{\bfn}}^{T}e^{it\log(n_{2}^{b}n_{3}^{a}/n_{1}^{a})}dt\Big\rvert,\end{equation} wherein the above sum runs over tuples satisfying the property that $n_{2}^{b/a}n_{3}$ is not an integer. We also note that Roth's theorem on rational approximation \cite{Roth} implies that for each pair $(n_{2},n_{3})$ with the property that $n_{2}^{b/a}$ is not an integer and for every fixed $\varepsilon>0$  then the inequality
\begin{equation}\label{Cep}\Big\lvert n_{2}^{b/a}n_{3}-N_{b,a}\Big\rvert\geq \frac{C'(\varepsilon,n_{2})}{n_{3}^{1+\varepsilon}}\end{equation} holds, where $C'(\varepsilon,n_{2})>0$ only depends on $\varepsilon$ and $n_{2}$. Therefore, the above estimate in conjunction with (\ref{logilogi}) delivers the lower bound 
\begin{equation}\label{Croth}\big\lvert \log \big(n_{2}^{b}n_{3}^{a}/N_{b,a}^{a}\big)\big\rvert\geq \frac{C(\varepsilon,n_{2})}{n_{3}^{2+\varepsilon}}.\end{equation} For the purpose of organising our argument rather neatly it seems pertinent to denote $W_{1}(T)$ the contribution to $J_{2,2}(T)$ of tuples satisfying $C(\varepsilon,n_{2})^{-1}n_{3}^{2+\varepsilon}\leq T$. Likewise, we write $W_{2}(T)$ for the contribution of tuples with the property that $C(\varepsilon,n_{2})^{-1}n_{3}^{2+\varepsilon}> T.$ Then by the preceding discussion it transpires that
$$W_{1}(T)\ll \sum_{C(\varepsilon,n_{2})^{-1}n_{3}^{2+\varepsilon}\leq T}C(\varepsilon,n_{2})^{-1}n_{2}^{-1/2-b/2a}n_{3}^{1+\varepsilon}\ll T\sum_{n_{2}\leq \sqrt{bT_{1}}}n_{2}^{-1/2-b/2a}\ll T,$$ where we estimated the integral in (\ref{J22ef}) by the inverse of the corresponding logarithm and we employed (\ref{Croth}) appropiately. 

In order to make progress we find it desirable to introduce the parameter $N$ and write $$C_{\varepsilon}(N)=\min_{1\leq n_{2}\leq N}C(\varepsilon,n_{2})$$ for further convenience. We then estimate the integral on the right side of (\ref{J22ef}) by the trivial bound $T$ and thus obtain
$$W_{2}(T)\ll T\sum_{C(\varepsilon,n_{2})^{-1}n_{3}^{2+\varepsilon}> T}n_{2}^{-1/2-b/2a}n_{3}^{-1}\ll T\big(W_{2,1}(T)+W_{2,2}(T)\big),$$wherein we splitted the above sum accordingly, namely $$W_{2,1}(T)=\sum_{\substack{n_{2}\geq N\\ n_{3}\leq \sqrt{aT_{1}}}}n_{2}^{-1/2-b/2a}n_{3}^{-1}\ \ \ \ \ \ \ \text{and}\ \ \ \ \ \ \ \ W_{2,2}(T)=\sum_{(C_{\varepsilon}(N)T)^{1/(2+\varepsilon)}< n_{3}\leq \sqrt{aT_{1}}}n_{3}^{-1}.$$ 
By summing over $n_{2}$ and $n_{3}$ we obtain the bound
\begin{equation}\label{L21L21}W_{2,1}(T)\ll N^{1/2-b/2a}\log T\end{equation}
Likewise, for fixed $0<\varepsilon\leq 1$ one finds that
\begin{equation}\label{L22L22}W_{2,2}(T)\ll \varepsilon\log T+\log C_{\varepsilon}(N).\end{equation}

It is of great importance to emphasize that the implicit constants cognate to the above bounds for $W_{2,1}(T)$ and $W_{2,2}(T)$ do not depend on neither $\varepsilon$ nor $N$. We also observe that in view of the analysis of $W_{2,2}(T)$ in conjunction with a careful perusal of the underlying argument underpinning the choice of the above cutoff parameters, it transpires that the presence of the exponent $2+\varepsilon$ in (\ref{Cep}) is essential. Therefore, by the preceding discussion we obtain for any fixed $\varepsilon>0$ the estimate
$$\lim_{T\to\infty}\frac{\lvert J_{2}(T)\rvert}{T\log T}\ll \varepsilon+N^{1/2-b/2a}.$$ Consequently, letting $N\to\infty$ and $\varepsilon\to 0$ in the above line and recalling that $a<b$ we obtain the desired result. If $a=1$ and $0<\theta<1$ then upon recalling (\ref{J22}) it is apparent that $J_{2,2}(T)=0$ since $n_{2}n_{3}$ is always an integer, such an observation in conjunction with (\ref{J212}) and (\ref{bJ21}) thereby delivering the desired conclusion.
\end{proof}

Before progressing in the proof we draw the reader's attention to the estimates (\ref{L21L21}) and (\ref{L22L22}) for the purpose of emphasizing on the fact that the ineffectiveness in Roth's theorem with respect to both $\varepsilon$ and $n_{2}$ is then transferred to the ineffectiveness of the error term cognate to the asymptotic formula deduced herein. We then find it desirable to combine the above lemmata to the end of obtaining the equation
$$I_{1}(T)\sim\frac{1}{2}\zeta\big((a+b)/2\big)T\log T,$$
and note that the application of Lemmata \ref{lem604} and \ref{lem6054} then yields $$\sum_{i=2}^{5}\lvert I_{i}(T)\rvert\ll T.$$ We finally use the observation that $$\chi(1/2+ait)\chi(1/2-ait)=1$$ to deduce the equality
\begin{align}\label{Y6a}Y_{6,a,b}(T)&
=\int_{0}^{T}D(1/2-ait)D(1/2-bit)D(1/2+ait)dt=I_{1}(T),
\end{align} where $Y_{6,a,b}(T)$ was defined right before Lemma \ref{lem606}. We also employ such a lemma to obtain the estimate
$$Y_{6,b,a}(T)\ll T^{5/4-b/4a}(\log T)+T^{3/4},$$ and find it adecquate to recall for further purposes the equation $$I_{6}(T)=Y_{6,a,b}(T)+Y_{6,b,a}(T)$$presented in (\ref{I66}). The combination of the above estimates and identities then delivers the required asymptotic formula for $I_{a,b,a}(T)$ and completes the proof of Theorem \ref{thm602}.

\section{The instance $c=1$}\label{hannah}
We shift our focus to the proof of Theorem \ref{loweuse}, it being worth anticipating that additional terms that were negligible on previous ocassions stemming from the analysis of $I_{5}(T)$ contribute in this setting to the main terms in this framework, a reappraisal of those computations thereby being required herein. We recall (\ref{piresiti}) and Lemma \ref{lem604} to deduce that in this context it also follows that $$\sum_{i=2}^{4}\lvert I_{i}(T)\rvert\ll T^{3/4}.$$ The upcoming proposition shall be devoted to evaluate $I_{5}(T)$, it being pertinent to introduce first
\begin{equation}\label{nunu}\nu_{b}=\zeta\big((b+1)/2\big)\big(2\gamma-1-\log 2\pi\big)+b\zeta'\big((b+1)/2\big),\ \ \ \ \ \ \gamma_{1}=\frac{1}{2}\lim_{s\to 1}((s-1)\zeta(s))''.\end{equation} It also seems worth  defining the universal parameters
\begin{equation}\label{c1c2c3} c_{1}=3\gamma-1-\log(2\pi),\ \ \ \ \ c_{0}=3\gamma^{2}-3\gamma+3\gamma_{1}+1+(1-3\gamma)\log 2\pi+(\log 2\pi)^{2}/2,\end{equation} wherein $\gamma$ is the Euler's constant. 

\begin{lem}\label{lemota}
Let $a=c=1$ and $b>1$. Then with the above notation and recalling (\ref{I1TTT}) one has that
$$I_{5}(T)=\zeta\big((b+1)/2\big)T\log T+\nu_{b}T-2TS_{1,b}(T)+2J_{4}(T)+O\big(T^{3/4}\log T+T^{1/2+1/2b+\varepsilon}\big).$$ If on the contrary $a=b=c=1$ then it follows that
$$I_{5}(T)=\frac{T(\log T)^{2}}{2}+c_{1}T\log T+c_{0}T-3TS_{1,1}(T)+3J_{4}(T)+O\big(T^{3/4}\log T\big).$$
\end{lem}
\begin{proof}
We shall begin our discussion by focusing first on the instance $b>1$ and draw the reader's attention to the statement of Lemma \ref{lemot} to the end of observing that then, when $c=1$ one has that
$$\frac{1}{2\pi} K_{1}(T)=\mathop{{\sum_{\substack{n_{1}\leq n_{2}^{b}n_{3}\\ n_{1}n_{2}^{b}n_{3}\leq T_{1}\\ n_{3}\leq n_{2}^{b}n_{1}}}}^*}n_{2}^{b/2-1/2}=\mathop{{\sum_{\substack{n_{1}n_{2}^{b}n_{3}\leq T_{1}\\ n_{3}\leq n_{2}^{b}n_{1}}}}^*}n_{2}^{b/2-1/2}-\sum_{\substack{n_{2}^{b}n_{3}< n_{1}\leq \sqrt{T_{1}}}}n_{2}^{b/2-1/2},$$ wherein upon recalling (\ref{Babc}) the corresponding triples in the sum with * satisfy the proviso $(n_{1},n_{2},n_{3})\in \mathcal{B}_{1,b,1}$, an application of the same argument in conjunction with the underlying symmetry delivering
\begin{equation*}\label{KK1}\frac{1}{2\pi}K_{1}(T)=\sum_{\substack{n_{1}n_{2}^{b}n_{3}\leq T_{1}\\ (n_{1},n_{2},n_{3})\in \mathcal{B}_{1,b,1}}}n_{2}^{b/2-1/2}-2\sum_{\substack{n_{2}^{b}n_{3}< n_{1}\leq \sqrt{T_{1}}}}n_{2}^{b/2-1/2}.\end{equation*}

It also seems worth observing that one may rewrite the first summand in the above equation by means of 
\begin{align*}\sum_{\substack{n_{1}n_{2}^{b}n_{3}\leq T_{1}\\ (n_{1},n_{2},n_{3})\in \mathcal{B}_{1,b,1}}}n_{2}^{b/2-1/2}=&\sum_{\substack{n_{1}n_{2}^{b}n_{3}\leq T_{1}}}n_{2}^{b/2-1/2}-2\sum_{\substack{\sqrt{T_{1}}< n_{1}\leq T_{1}/(n_{2}^{b}n_{3})}}n_{2}^{b/2-1/2},\end{align*}whence combining the above equations delivers
\begin{equation}\label{KKIO}\frac{1}{2\pi}K_{1}(T)=\sum_{\substack{n_{1}n_{2}^{b}n_{3}\leq T_{1}}}n_{2}^{b/2-1/2}-2\sum_{\substack{n_{2}^{b}n_{3}< n_{1}\leq T_{1}/(n_{2}^{b}n_{3})}}n_{2}^{b/2-1/2}.\end{equation}
We find it pertinent to focus on the last summand in the previous line for the purpose of noting that then 
\begin{align*}\sum_{\substack{n_{2}^{b}n_{3}< n_{1}\leq T_{1}/(n_{2}^{b}n_{3})}}n_{2}^{b/2-1/2}=&T_{1}\sum_{\substack{n_{2}^{b}n_{3}\leq \sqrt{T_{1}}}}n_{2}^{-(b+1)/2}n_{3}^{-1}-\sum_{\substack{n_{2}^{b}n_{3}\leq \sqrt{T_{1}}}}n_{2}^{(3b-1)/2}n_{3}
\\
&+O\Big(\sum_{\substack{n_{2}^{b}n_{3}\leq \sqrt{T_{1}}}}n_{2}^{b/2-1/2}\Big).
\end{align*}
The reader may observe that a routinary argument then delivers the bound
$$\sum_{\substack{n_{2}^{b}n_{3}\leq \sqrt{T_{1}}}}n_{2}^{b/2-1/2}\ll T^{1/4+1/4b}\sum_{n_{3}\leq \sqrt{T_{1}}}n_{3}^{-1/2-1/2b}\ll T^{1/2},$$whence recalling (\ref{I1TTT}) once again to the reader and combining the preceding formulae yields
$$\sum_{\substack{n_{2}^{b}n_{3}< n_{1}\leq T_{1}/(n_{2}^{b}n_{3})}}n_{2}^{b/2-1/2}=T_{1}S_{1,b}(T)-\frac{1}{2\pi}J_{4}(T)+O(T^{1/2}).$$

The perusal of the first summand in (\ref{KKIO}) entails utilising the classical asymptotic formula for the Dirichlet problem, namely
\begin{align*}\sum_{\substack{n_{1}n_{2}^{b}n_{3}\leq T_{1}}}n_{2}^{b/2-1/2}=&\sum_{n_{2}^{b}\leq T_{1}}n_{2}^{b/2-1/2}\Big(\Big\lfloor T_{1}/n_{2}^{b}\Big\rfloor\log \Big\lfloor T_{1}/n_{2}^{b}\Big\rfloor+(2\gamma-1)\Big\lfloor T_{1}/n_{2}^{b}\Big\rfloor\Big)
\\
&+O\Big(T^{1/2}\sum_{n_{2}^{b}\leq T_{1}}n_{2}^{-1/2}\Big).
\end{align*}
It then transpires by using a routine argument that then
\begin{align*}\sum_{\substack{n_{1}n_{2}^{b}n_{3}\leq T_{1}}}n_{2}^{b/2-1/2}=&\big(T_{1}(\log T_{1}+2\gamma-1)\big)\sum_{n_{2}^{b}\leq T_{1}}n_{2}^{-(b+1)/2}-bT_{1}\sum_{n_{2}^{b}\leq T_{1}}n_{2}^{-(b+1)/2}\log n_{2}
\\
&+O\Big(\log T\sum_{n_{2}^{b}\leq T_{1}}n_{2}^{b/2-1/2}\Big)+O(T^{1/2+1/2b}),
\end{align*}
whence rearranging the terms appropiately then enables one to conclude that
$$2\pi\sum_{\substack{n_{1}n_{2}^{b}n_{3}\leq T_{1}}}n_{2}^{b/2-1/2}=\zeta\big((b+1)/2\big)T\log T+\nu_{b}T+O\big(T^{3/4}\log T\big),$$ as desired. Therefore, the combination of the preceding formulae in conjunction with Lemma \ref{lemot} yields the desired result for the instance $b>1$.

We shift our attention to the case $b=1$ and employ a similar argument to that utilised for the former case. We thus draw the reader's attention to the statement of Lemma \ref{lemot} and observe that then
$$\frac{1}{2\pi}K_{1}(T)=\mathop{{\sum_{\substack{n_{1}n_{2}n_{3}\leq T_{1}\\ n_{1}\leq n_{2}n_{3}\\  n_{3}\leq n_{2}n_{1}\\ n_{2}\leq n_{1}n_{3}}}}^*}1=\mathop{{\sum_{\substack{n_{1}n_{2}n_{3}\leq T_{1}\\ n_{3}\leq n_{2}n_{1}\\ n_{2}\leq n_{1}n_{3}}}}^*}1-\sum_{\substack{n_{2}n_{3}< n_{1}\leq \sqrt{T_{1}}}}1,$$ wherein we omitted upon recalling (\ref{Babc}) writing the proviso $(n_{1},n_{2},n_{3})\in \mathcal{B}_{1,1,1}$, an iteration of the same argument combined with the underlying symmetry delivering
\begin{equation*}\label{KK1w}\frac{1}{2\pi}K_{1}(T)=\sum_{\substack{n_{1}n_{2}n_{3}\leq T_{1}\\ (n_{1},n_{2},n_{3})\in \mathcal{B}_{1,1,1}}}1-3\sum_{\substack{n_{2}n_{3}< n_{1}\leq \sqrt{T_{1}}}}1.\end{equation*}
It also seems worth noting as is customary that one may rewrite the first summand in the above equation by means of the expressions
\begin{align*}\sum_{\substack{n_{1}n_{2}n_{3}\leq T_{1}\\ (n_{1},n_{2},n_{3})\in \mathcal{B}_{1,1,1}}}1=&\sum_{\substack{n_{1}n_{2}n_{3}\leq T_{1}}}1-3\sum_{\substack{\sqrt{T_{1}}< n_{1}\leq T_{1}/(n_{2}n_{3})}}1,\end{align*}whence combining the above equations delivers
\begin{equation*}\label{KKIOP}\frac{1}{2\pi}K_{1}(T)=\sum_{\substack{n_{1}n_{2}n_{3}\leq T_{1}}}1-3\sum_{\substack{n_{2}n_{3}< n_{1}\leq T_{1}/(n_{2}n_{3})}}1.\end{equation*}The same argument utilised in the discussion cognate to the instance $b>1$ enables one to deduce upon recalling (\ref{I1TTT}) the formula
$$\sum_{\substack{n_{2}n_{3}< n_{1}\leq T_{1}/(n_{2}n_{3})}}1=T_{1}S_{1,1}(T)-\frac{1}{2\pi}J_{4}(T)+O\big(T^{1/2}\log T\big).$$ 

To complete the proof it just remains to observe that $$\sum_{\substack{n_{1}n_{2}n_{3}\leq T_{1}}}1=\sum_{n\leq T_{1}}d_{3}(n),$$wherein $d_{3}(n)$ denotes the number of ways of writing $n$ as a product of three positive integers, whence an application of Kolesnik \cite{Kol} yields
$$\sum_{\substack{n_{1}n_{2}n_{3}\leq T_{1}}}1=\frac{1}{2}T_{1}(\log T_{1})^{2}+(3\gamma-1)\log T_{1}+(3\gamma^{2}-3\gamma+3\gamma_{1}+1)T_{1}+O(T^{43/96+\varepsilon}).$$
The preceding discussion then in conjunction with Lemma \ref{lemot} delivers the desired result.
\end{proof}

The last lines of this section shall be devoted to combine the above lemmata to the end of deducing Theorem \ref{loweuse}. We employ first Lemma \ref{lem606} when $b>1$ to derive the estimate $$Y_{6,b,1}(T)\ll T^{3/4}\log T,$$ 
and combine the above discussion with the observation (\ref{Y6a}) and Lemma \ref{lem604} to obtain 
\begin{align*}I_{1,b,1}(T)=&2I_{1}(T)+\zeta\big((b+1)/2\big)T\log T+\nu_{b}T-2TS_{1,b}(T)+2J_{4}(T)
\\
&+O\big(T^{3/4}\log T+T^{1/2+1/2b+\varepsilon}\big),
\end{align*} the application of Lemma \ref{rothi} for the choice $\theta=1/2$ in conjunction with (\ref{I1TTT}) and the preceding line thus delivering the result in Theorem \ref{loweuse} concerning the instance $b>1$. We shift our attention to the case $b=1$ and recall (\ref{I66}) for the purpose of observing that then 
$$I_{6}(T)=2Y_{1,1,1}(T)=2I_{1}(T),$$ whence the preceding discussion combined with equation (\ref{piresiti}) and Lemmata \ref{lem604} and \ref{lemota} assures the validity of the asymptotic formula \begin{align*}I_{1,1,1}(T)=&3I_{1}(T)+\frac{T(\log T)^{2}}{2}+c_{1}T\log T+c_{0}T-3TS_{1,1}(T)+3J_{4}(T)
\\
&+O\big(T^{3/4}\log T\big).\end{align*} The above equation then combined with (\ref{I1TTT}) for the choice $\theta=1/2$ and Lemma \ref{rothi} yields the desired conclusion and completes the proof of Theorem \ref{loweuse}.

\end{document}